\documentclass[11pt]{amsart}

\RequirePackage[l2tabu,orthodox]{nag}
\usepackage[T1]{fontenc}
\usepackage[utf8]{inputenc}
\usepackage{lmodern}
\usepackage[english]{babel}
\usepackage{color} 
\usepackage[usenames,dvipsnames]{xcolor}
\usepackage[alphabetic]{amsrefs}
\usepackage{amsthm,amssymb,amsmath}
\usepackage{enumerate}
\usepackage[inline]{enumitem}
\usepackage{url}
\usepackage{multicol}
\usepackage{mathrsfs}  
\usepackage{tikz-cd}

\numberwithin{equation}{section}

\newtheorem{thm}{Theorem}[section]

\newtheorem{ques}[thm]{Question}

\newtheorem*{claim}{Claim}

\newtheorem{claimnum}{Claim}

\newtheorem{lem}[thm]{Lemma}
\newtheorem{prop}[thm]{Proposition}
\newtheorem{cor}[thm]{Corollary}

\theoremstyle{definition}
\newtheorem{dfn}[thm]{Definition}
\newtheorem{standing-assumption}[thm]{Assumption}

\theoremstyle{remark}
\newtheorem{rem}[thm]{Remark}
\newtheorem{ex}[thm]{Example}

\usepackage[margin=3cm]{geometry}

\usepackage{indentfirst}
\usepackage{microtype}

\usepackage[colorlinks=true,linkcolor=blue,citecolor=magenta]{hyperref}

\usepackage{blindtext}

\usepackage{nameref}

\DeclareMathOperator{\homeo}{\mathsf{Homeo}}

\DeclareMathOperator{\rep}{\mathsf{Rep}_{\operatorname{irr}}}

\newcommand{\fix}{\operatorname{\mathsf{Fix}}}
\DeclareMathOperator{\BP}{\mathsf{BP}}

\newcommand{\id}{\mathsf{id}}
\newcommand{\R}{\mathbb{R}}

\newcommand{\N}{\mathbb N}
\newcommand{\Z}{\mathbb Z}

\newcommand{\Ccal}{\mathcal{C}}

\newcommand{\nice}{{\mathfrak{D}}}
\newcommand{\Tsf}{\mathsf{T}}
\newcommand{\Hsf}{\mathsf{H}}

\newcommand{\Fsing}{\mathscr{F}}
\newcommand{\Tsingsf}{\mathscr{T}}

\newcommand{\Tsing}{\mathscr{T}}

\renewcommand{\setminus}{\smallsetminus}

\renewcommand{\emptyset}{\varnothing}

\DeclareMathOperator{\supp}{\mathsf{Supp}}

\title{A realisation result for moduli spaces of group actions on the line}

\date{\today}
\author{Joaquín Brum}
\author{Nicol\'as Matte Bon} 
\author{Crist\'obal Rivas}
\author{Michele Triestino}

\thanks{All the authors acknowledge the support of the projects MATH AMSUD, DGT
-- Dynamical Group Theory (22-MATH-03) and  ANR Gromeov (ANR19-CE40-0007). The work of NMB was supported by the LABEX MILYON (ANR-10-LABX-0070) of Université de Lyon, within the program ```France 2030'' (ANR-11-IDEX-0007) operated by the French National Research Agency (ANR). 
MT has been partially supported by the project ANER Agroupes (AAP 2019 Région Bourgogne Franche–Comté) and his host department IMB receives support from the EIPHI Graduate
School (ANR-17-EURE-0002).}
\keywords{Group actions on the real line, semi-conjugacy of actions, Deroin space}
\begin{document}
\maketitle

\begin{abstract}
	Given a finitely generated group $G$, the possible actions of $G$ on  the real line (without global fixed points), considered up to semi-conjugacy, can be encoded by the space of orbits of a flow on a compact space $(Y, \Phi)$ naturally associated with $G$ and uniquely defined up to flow equivalence, that we call the \emph{Deroin space} of $G$. We show a realisation result:  every expansive flow $(Y, \Phi)$ on a compact metrisable space of topological dimension 1, satisfying some mild additional assumptions,  arises as the Deroin space of a finitely generated group. This is proven by identifying the Deroin space of an explicit family of groups acting on suspension flows of subshifts, which is a variant of a construction introduced by the second and fourth authors. This result provides a source of examples of finitely generated groups satisfying various new phenomena for actions on the line, related to their rigidity/flexibility properties and to the structure of (path-)connected components of the space of actions.

			\smallskip
	
	{\noindent\footnotesize \textbf{MSC\textup{2020}:} Primary 37C85, 57M60. Secondary 37E05, 37B05.}
\end{abstract}
%\tableofcontents	

\section{Introduction}

The topic of this paper is the study of actions of finitely generated groups $G$ on the real line, that is, representations of $G$ into the group $\homeo_0(\R)$ of  orientation-preserving homeomorphisms of the line.  We start with some basic terminology. A representation $\rho\colon G\to \homeo_0(\R)$ is said to be \emph{irreducible} if $\rho(G)$ has no fixed points in $\R$. We will denote by $\rep(G;\R)$ the space of irreducible representations of a group $G$ in $\homeo_0(\R)$. The latter has a natural  (Polish) topology, inherited from the compact-open topology on $\homeo_0(\R)$.
One says that two representations $\rho_1, \rho_2\in \rep (G;\R)$ are  \emph{semi-conjugate} if there exists a non-decreasing $G$-equivariant map $h\colon \R \to \R$ (here we strictly reserve this terminology to the case of a non-decreasing map, and will always specify when two representations are semi-conjugate only up to reversing the orientation). This defines an equivalence relation on $\rep(G; \R)$. Considering irreducible representations up to semi-conjugacy is a way to restrict the attention to their minimal sets, which in a sense are the basic building blocks of the dynamics for group actions in dimension 1. 

Our main focus is the study of the semi-conjugacy classes in $\rep(G; \R)$ from a topological and dynamical perspective. One might think of the quotient of $\rep(G; \R)$ by the semi-conjugacy relation as a {moduli space} for actions of $G$ on the line; unfortunately this quotient is often too ill-behaved  to develop a meaningful theory.  In fact for many groups $G$ the semi-conjugacy relation is not \emph{smooth} in the Borel sense, namely there is no Borel map from $\rep(G; \R)$ to any standard Borel space which is constant on semi-conjugacy classes and separates them. This can be seen as an obstruction to find explicit invariants that classify actions up to semi-conjugacy. This contrasts with the case of actions on the circle, for which complete semi-conjugacy invariants are available (see Ghys \cite{Ghys} and Matsumoto \cite{Matsumoto}) and  the quotient by the semi-conjugacy relation is a compact metrisable space (see Mann and Wolff \cite{MW}).  

For finitely generated groups, a solution to this problem is to consider an intermediate quotient space, that we call the \emph{Deroin space} of the group $G$,  which is a canonical version of a construction  of Deroin \cite{Deroin}. It is a compact metrisable space $(\mathcal{D}, \Psi)$ endowed with a  one-parameter {flow} $\Psi\colon\R\times \mathcal D\to \mathcal D$ and a continuous $G$-action preserving $\Psi$-orbits, which is uniquely associated to $G$ up to $G$-equivariant flow equivalence (that is, homeomorphism sending flow-orbits to flow-orbits in an orientation preserving way). The space $\mathcal{D}$ can be abstractly defined as the quotient of $\rep(G; \R)$ by the finer equivalence relation of \emph{pointed}  semi-conjugacy, namely semi-conjugacy preserving a marked point on the line (see Definition \ref{dfn.pointedsemi}), and the flow $\Psi$ moves the marked point. By the results of Deroin, Kleptsyn, Navas, and Parwani \cite{DKNP}, the space $\mathcal{D}$ can actually be realised as a subspace of $\rep(G; \R)$ by selecting a suitable system of representatives (namely the set of normalized $\mu$-harmonic representations), in such a way that the flow $\Psi$ coincides with the  restriction to $\mathcal{D}$ of the translation action of $(\R, +)$ on $\rep(G; \R)$. (See \S\ref{sc:Deroin} for details). 

The interest for Deroin spaces from the perspective described above is that the quotient map $r_{\mathcal{D}}\colon \rep(G; \R)\to \mathcal{D}$ is continuous, surjective, and has the property that two actions $\rho_1, \rho_2\in \rep(G; \R)$ are semi-conjugate if and only if $r_{\mathcal{D}}(\rho_1), r_{\mathcal{D}}(\rho_2)$ belong to the same $\Psi$-orbit. Therefore all the complexity of the semi-conjugacy relation in $\rep(G; \R)$ is witnessed by the dynamics of the flow $(\mathcal{D}, \Psi)$.  This is a strong restriction even if one retains only the Borel structure of $\rep(G; \R)$, see the discussion in our previous work \cite[\S 14.4]{BMRT} (in particular the semi-conjugacy relation on $\rep(G; \R)$ is smooth if and only if the orbit-equivalence relation of $(\mathcal{D}, \Psi)$ is so).
Well beyond, the flow $(\mathcal{D}, \Psi)$ carries  topological information on how semi-conjugacy classes sit inside $\rep(G; \R)$.

Describing the Deroin space of a finitely generated group requires a global understanding of its possible actions on the line up to semi-conjugacy. The pool of examples for which such understanding is available gives rise to rather degenerate flows; let us mention some.  It has long been known  that if $G$ is a group of subexponential word-growth, then every $\rho\in \rep(G; \R)$ is semi-conjugate to an action by translations (see Plante \cite{Plante-measure}); this implies that its Deroin space is either empty or homeomorphic to a sphere $\mathbb{S}^n$ with trivial flow.  Some groups admit finitely many actions on the line up to semi-conjugacy, such as the solvable Baumslag--Solitar groups (see the work of the third author \cite{Rivas-conradian}), or central extensions of some groups acting on the circle such as Thompson's group $T$ (see the work of the second and fourth authors \cite{MB-T}), or the Fuchsian group $\Delta(2, 3, 7)$ (see Mann and the fourth author \cite{Mann-Triestino}); in  these cases the flow ($\mathcal{D}, \Psi)$ can be easily computed and has only finitely many orbits (e.g.\ in the last two examples it is the disjoint union of two circles, which are $\Psi$-periodic orbits). Some more sources of examples come from our previous works \cite{BMRT, solvable}, and give rise to flows with an uncountable number of orbits but still very restricted dynamics: for example in the case of Thompson's group $F$,   the Deroin space has a smooth orbit space and all orbits converge towards some limit point  (see \cite[Figure 1.3.1]{BMRT});  the same  holds for the wreath product $\Z \wr \Z$, relying on \cite[Proposition 6.1]{solvable} (a more precise result for all finitely generated solvable groups will be the object of a work in preparation). An example on the opposite side of the spectrum is given by the non-abelian free groups  $\mathbb{F}_n$. In this case the flow $(\mathcal{D}, \Psi)$  has a non-smooth space of orbits, as one can show by exhibiting a family of representations of $\mathbb{F}_n$ giving rise to a subflow with chaotic behaviour \cite[Remark 14.4.4]{BMRT}. An explicit description of the Deroin space of $\mathbb{F}_n$ is not known (note that it would encode actions on the line of every $n$-generated group). 

Our main result shows that many flows can be realised as the  Deroin space of a finitely generated group. This is proven by identifying the Deroin space of an explicit family of groups, which are relatives of the groups introduced by two of the authors in \cite{MB-T}.  The class of flows that we consider are suspension flows of subshifts,  intrinsically characterised (by Bowen and Walters \cite{expansive}) as expansive flows on compact metrisable spaces of topological dimension 1 (all these notions are recalled in \S \ref{s-Deroin-identification}). We say that a flow $(Y, \Phi)$ is \emph{freely reversible} if it admits a fixed-point-free  order-2 homeomorphism  $\sigma \colon Y\to Y$  which is a flow equivalence between $\Phi$ and its time-reverse; the Deroin space of a finitely generated group always has this property (Proposition \ref{p-Deroin-freely-reversible}).
\begin{thm}\label{t-realisation}
	Let $(Y, \Phi)$ be an expansive flow without fixed points on a compact metrisable space of topological dimension 1. Assume that $\Phi$ is freely reversible and topologically free (namely, the set of  non-periodic points is dense).  Then, there exists a finitely generated group $G$ acting faithfully on $Y$ by homeomorphisms,   preserving each $\Phi$-orbit,  such that $(Y,\Phi)$ can be identified with the Deroin space of $G$ (up to $G$-equivariant flow equivalence).
	\end{thm}
	
	\begin{rem}
		Recall that a group is \emph{left-orderable} if it admits a total order which is invariant under left multiplication. A countable group is left-orderable if and only if it is isomorphic to a subgroup of $\homeo_0(\R)$. Since we assume in Theorem \ref{t-realisation} that $\Phi$ is  topologically free, the group $G$ acts faithfully on the union of non-periodic orbits, and thus it is left-orderable by a standard argument (see \cite[Proposition 3.3]{MB-T}).
	\end{rem}
	
	Theorem \ref{t-realisation} provides the first explicit computation of Deroin spaces with chaotic behaviour (in particular whose space of orbits is non-smooth). In fact it produces examples such that the flow $(\mathcal{D}, \Psi)$ has rather arbitrary dynamical properties.
%	; this is a source of new phenomena for the behaviour of representations in $\rep(G; \R)$.
		
One direction of application of Theorem \ref{t-realisation} is related  to the study of rigidity and flexibility of group actions on the line. In a broad sense, this topic aims at understanding, given a group $G$, how its representations can be modified (up to semi-conjugacy) by slight perturbations, which boils down to understand  how semi-conjugacy classes accumulate onto each other.  This leads to several tightly related notions of rigidity and flexibility for representations, appearing in the literature on group actions on one-manifolds, see e.g.\  \cite{Mann-survey, Mann-Rivas, MW, KKMj}.   We will use here the following terminology, analogous to the one used by Mann and Wolff \cite{MW} for actions on the circle.
	\begin{dfn} A representation $\rho \in \rep(G; \R)$ is \emph{locally rigid} if it is an interior point in its semi-conjugacy class. 
It is \emph{rigid} if its whole semi-conjugacy class is open.  A representation which is not locally rigid is called \emph{flexible}. \end{dfn}
	A representation $\rho\in \rep(G; \R)$ is rigid if and only if its image $r_\mathcal{D}(\rho)$ in the Deroin space $(\mathcal{D}, \Psi)$ defines an isolated point in any local cross-section of the flow $\Psi$, i.e.\ if its $\Psi$-orbit is an an open subset of $\mathcal{D}$. Local rigidity is not witnessed by the Deroin space in general, although it is equivalent to rigidity for minimal representations \cite[Remark 14.2.5]{BMRT}. However for the groups that we construct to prove Theorem \ref{t-realisation}, we additionally show that local rigidity and rigidity are equivalent for all representations  (see Corollary \ref{c-rigid}). Thus for this family of groups we have complete control on which representations are (locally) rigid (clearly a flow as in Theorem \ref{t-realisation} may or may not admit open orbits). More generally,  given $\rho\in \rep(G; \R)$, we obtain a precise characterisation of all representations that accumulate on $\rho$, up to semi-conjugacy, in terms of the flow $(Y, \Phi)$, see Theorem \ref{t-accumulation}.	
	
	A relevant special case of Theorem \ref{t-accumulation} arises by choosing the flow $(Y, \Phi)$ to be \emph{minimal}, i.e.\ with all orbits dense. In this case our construction produces examples of groups with the following extreme flexibility property.
	\begin{cor}\label{c-dense}
	There exist finitely generated groups $G$ for which any semi-conjugacy class is dense in $\rep(G; \R)$ (and $\rep(G, \R)$ contains uncountably many semi-conjugacy classes). \end{cor}
		
		In other words, a group $G$ as in Corollary \ref{c-dense} has the property that every representation $\rho\in \rep(G; \R)$ admits arbitrarily small perturbations which are semi-conjugate to any other representation in $\rep(G; \R)$. We call such a representation $\rho$  \emph{universally flexible} (see Definition \ref{s-universally-flexible}). To our knowledge, there were previously no groups known to admit any representation with this property.

		 \begin{rem}
		 Corollary \ref{c-dense} can be compared with the problem of whether there exists a finitely generated left-orderable group acting minimally on its space of left orders, see Navas \cite[Question 6]{Navas-ICM}.
		 The conclusion in Corollary \ref{c-dense} is the analogous property at the level of actions on the line;  however there is no formal connection: the groups constructed here do not act minimally on their space of left-orders, and conversely this property does not seem to imply \emph{a priori} the conclusion of Corollary \ref{c-dense} (because not all representations in $\rep(G;\R)$ on the line are dynamical realisations of orders). However both properties imply that the group $G$ acts minimally on its Deroin space, which can be identified with a quotient of the space of left-orders \cite[Chapter 14]{BMRT}.
		 \end{rem}
		 
		  Beyond the minimal case, different choices for the flow $(Y, \Phi)$  yield groups with other prescribed  rigidity/flexibility property, for example groups  which admit both universally flexible representations together with  a given number (finite or countably infinite) of rigid ones (see Corollary \ref{c-prescribed-rigidity}). 
		  
		  In a similar spirit, we consider a natural analogue of the Cantor--Bendixson rank (CB-rank) for the space of semi-conjugacy classes in $\rep(G; \R)$, that we call the \emph{semi-conjugacy CB-rank} of $\rep(G; \R)$, defined according to a transfinite process in which the open semi-conjugacy classes play a role analogous to the isolated points for the usual CB-rank  (Definition \ref{d-CB-rank}). Then Theorem \ref{t-realisation}  produces examples of groups for which the semi-conjugacy CB-rank of $\rep(G; \R)$ can be computed and can be an arbitrarily large countable ordinal, see Corollary \ref{c-CB-rank}. Again, this can be seen as a counterpart at the level of actions of a question asked by Mann and the third author \cite{Mann-Rivas} for spaces of left orders, namely whether there exists a finitely generated left-orderable group whose space of left orders has Cantor--Bendixson rank larger than 1.

As another application, we obtain examples of finitely generated groups that admit many non semi-conjugate actions on $\R$, yet have a unique \emph{generic} action (up to possibly reversing the orientation) in the following strong sense.
\begin{cor}[Groups with a generic representation] \label{c-generic}
There exist finitely generated groups $G$  such that $\rep(G; \R)$ contains infinitely (countably or uncountably) many distinct semi-conjugacy classes, and there is $\rho\in \rep(G; \R)$ such that the semi-conjugacy classes of $\rho$ and of its conjugate by the reflection $x\mapsto -x$ form a dense open set. 
\end{cor}
\begin{rem}
Since we required semi-conjugacy to be orientation-preserving, it is not possible for a group $G$ to have a representation $\rho\in \rep(G; \R)$ with an open dense (or even co-meager) semi-conjugacy class, as its conjugate by the reflection should have the same property.
\end{rem}

		Another general theme in the study of representation spaces is the classification of their connected and path-connected components. 
		See for instance the works of Mann and Wolff \cite{Mann, MW} for the case of surface group actions on the circle (for which this problem remains open in general). For a finitely generated group $G$, the semi-conjugacy classes in $\rep(G; \R)$ are always path connected, so the connected and path-connected components of $\rep(G; \R)$ are in correspondence with those of the Deroin space (Proposition \ref{p-deroin-correspondence}). Again in analogy  with the terminology of \cite{MW}, one may call an action $\rho\in \rep(G; \R)$ \emph{path-rigid} if its path-component coincides with its semi-conjugacy class (that is, $\rho$ cannot be  deformed \emph{along a continuous path} into any non-semi-conjugate representation). 
Theorem \ref{t-realisation}  is a source of groups with a rather extreme behaviour with regard to this notion. In particular, it shows that  path-continuous deformations of representations  might be dramatically more restricted than arbitrary small perturbations in the compact-open topology.

\begin{cor}[Path-continuous deformations vs general perturbations] \label{c-path-rigid}
For every group $G$ satisfying the conclusion of Theorem \ref{t-realisation}, the path-components of  $\rep(G; \R)$ coincide with the semi-conjugacy classes.   

Combining with Corollary \ref{c-dense}, there are finitely generated groups $G$ such that all representations in $\rep(G; \R)$ are simoultaneously path-rigid and universally flexible, and the space $\rep(G; \R)$  is connected but not path-connected, and nowhere locally connected.
\end{cor}

	\subsection*{On the group construction}
	The groups that we consider to prove Theorem \ref{t-realisation} are based on a modification of the groups $\Tsf(\varphi)$ acting on the suspension space of subhifts from \cite{MB-T}.  A {subshift} $(X, \varphi)$ is a closed shift-invariant subset $X$ of $A^\Z$, where $A$ is a finite alphabet, and $\varphi\colon X\to X$ denotes the restriction of the shift.  Subshifts are characterised up to topological conjugacy as expansive homeomorphisms of compact, totally disconnected, metrisable spaces. The  \emph{suspension space} of $(X, \varphi)$ is the space $Y:=(X\times \R)/_{(\varphi(x),t)\sim(x,t+1)}$.  The space $Y$ is endowed with the \emph{suspension flow} $\Phi\colon \R\times Y\to Y$, induced by the natural flow on $X\times \R$ which translates the $\R$-coordinate. To take into account the restriction that the Deroin space is freely reversible, we need to consider {reversible} subshifts, by which we mean the data  $(X,\varphi,\sigma)$ of a subshift additionally equipped with an involution $\sigma \colon X\to X$ (an order-2 homeomorphism) such that $\sigma \varphi \sigma=\varphi^{-1}$. Such an involution naturally induces an involution $\hat{\sigma}\colon Y\to Y$ on the suspension space, which reverses the suspension flow. To ensure that $\hat{\sigma}$ has no fixed points, we need to further assume that $\sigma$ does not preserve any $\varphi$-orbit (see Lemma \ref{l-flipping-involution}). Every flow as in Theorem \ref{t-realisation} can be realised as the suspension flow of a subshift $(X, \varphi)$ with these properties, relying on Bowen and Walters \cite{expansive}. 
	
	 Given $(X, \varphi)$, the group $\Tsf(\varphi)$ is defined as a group of homeomorphisms of the its suspension space $Y$  analogous to Thompson's groups, whose action in the flow direction is by dyadic PL homeomorphisms, and such that the displacement along each flow-orbit is locally constant in the transverse direction. For a reversible subshift $(X, \varphi, \sigma)$ one can define an analogous group $\Tsf(\varphi, \sigma)$ as the centralizer of $\hat{\sigma}$ in $\Tsf(\varphi)$ (this definition appears in the work by Le Boudec and the second author \cite{confined}); a special case of the groups $\Tsf(\varphi, \sigma)$ are the groups defined by Hyde and Lodha in \cite{HL}.  The definition of the groups $\Tsf(\varphi)$ was largely inspired by the notion of Deroin space, and since their introduction the question naturally arose whether such constructions could be used  to prove a realisation result as Theorem \ref{t-realisation}, by identifying their Deroin space completely (see Question 1 in \cite{MB-T}). With the original definition of the groups $\Tsf(\varphi)$ or $\Tsf(\varphi, \sigma)$, our attempts to do so have run into  technical difficulties,  related to the fact that certain subgroups  playing an important role (namely the subgroups supported in dyadic charts)  have huge abelianisation (free abelian of infinite rank). This complicates the analysis of actions on the line. This technical point is  an artifact of the use of PL maps in the definition of the groups  $\Tsf(\varphi)$. A starting point of this paper is the idea {of replacing} PL maps by a larger pseudogroup of transformations; namely we consider dyadic PL transformations where we allow a non-discrete set of discontinuity points for the derivative, which accumulate onto some isolated ``higher order'' singularities with a controlled behaviour, required to locally commute with the doubling map.  Certain groups acting by such countably singular PL maps were already considered in our previous work \cite[\S 10.4]{BMRT}. On the interval $[0, 1]$, this definition yields a finitely generated group $\Fsing$ sharing many features with Thompson's group $F$,  except that it is \emph{perfect}.  Accordingly we define groups $\Tsing(\varphi)$ and $\Tsing(\varphi, \sigma)$ associated with any reversible subshift $(X, \varphi, \sigma)$. The main content of this work is the identification of the Deroin space of the groups $\Tsing(\varphi, \sigma)$ with the suspension flow of $(X, \varphi)$, from which Theorem \ref{t-realisation} follows. 
	
	\begin{rem}
		Beyond finite generation, the other properties established for the groups $\Tsf(\varphi)$ in \cite{MB-T} remain valid for the groups $\Tsing(\varphi)$ and $\Tsing(\varphi,\sigma)$, with similar proofs. For instance, these groups are simple if and only if $\varphi$ is minimal \cite[Theorem B]{MB-T}, they do not have Kazhdan's property ($T$) \cite[Theorem F]{MB-T}, and they are not finitely presented if $\varphi$ is not a subshift of finite type \cite[Theorem G]{MB-T}. We avoid a more detailed discussion here, as these properties are not relevant for the main object of the paper.
	\end{rem}
	
It would be interesting to identify similar constructions of groups acting {on flows over spaces} of higher topological dimension, and use them to extend Theorem \ref{t-realisation}.  Thus we conclude this introduction by proposing the following general question.
	
	\begin{ques}
	Which compact spaces $(Y, \Phi)$ endowed with a flow can be realised as the Deroin space of a finitely generated group?
	\end{ques}

\subsection*{Acknowledgements} We thank Ville Salo for sharing useful comments on the CB-rank of subshifts.

\section{General preliminaries on actions on the line} Recall that given a group $G$, we  call a representation $\rho\colon G\to \homeo_0(\R)$  irreducible if it has no fixed points. The space of all irreducible representations is denoted by $\rep(G; \R)$. It is endowed with the topology induced from the product topology on $\homeo_0(\R)^G$, where $\homeo_0(\R)$ has the compact-open topology.
We say that $\rho_1, \rho_2\in\rep(G; \R)$ are \emph{semi-conjugate} if there exists a non-decreasing map $h\colon \R \to \R$ which intertwines $\rho_1$ and $\rho_2$. Semi-conjugacy is an equivalence relation. 

Throughout the text, by a \emph{minimal set} for  a representation $\rho\in \rep(G; \R)$ we mean a closed, non-empty, minimal invariant subset $\Lambda\subset \R$. Every such set satisfies one of the following possibilities:
	\begin{enumerate}
		\item either $\Lambda$ is a closed discrete orbit, in which case the action of $G$ on $\Lambda$ factors through a cyclic quotient of $G$, or
		\item $\Lambda$ is a perfect set and it is the unique minimal set of $\rho$. Furthermore, in this case we have either $\Lambda=\R$, or $\Lambda$ has empty interior.
	\end{enumerate}

It is also well known that when $G$ is \emph{finitely generated}, every irreducible action $\rho\in \rep(G;\R)$ admits a minimal set, see for instance Navas \cite[Proposition 2.1.12]{Navasbook}. (However such a set need not exist if the finite generation assumption is dropped.)

This discussion implies that if $G$ is a finitely generated group, then every semi-conjugacy class in $\rep(G; \R)$ contains a representative which is either minimal, or \emph{cyclic} (namely which factors through an epimorphism of $G$ onto a cyclic group of translations). Such a representative is unique up to conjugacy by an element of $\homeo_0(\R)$. Moreover any map implementing a semi-conjugacy between minimal actions is automatically a homeomorphism, and hence a conjugacy  (see for instance Kim, Koberda, and Mj \cite[Lemma A.4]{KKMj}).

We also record the following. 
\begin{prop} \label{p-path-connected-semiconjugacy}
Let $G$ be a finitely generated group. Then every semi-conjugacy class in $\rep(G; \R)$ is path-connected. 
\end{prop}
\begin{proof}
Let $\mathcal{C}\subset \rep(G; \R)$ be a semi-conjugacy class, and choose a representative $\rho\in \mathcal{C}$ which is minimal or cyclic. It is enough to see that there exists a continuous path from any $\tilde\rho\in \mathcal{C}$ to ${\rho}$ inside $\mathcal{C}$.
	Assume first that $\rho$ is minimal. Let $h:\R\to \R$ be a semi-conjugacy between $\tilde{\rho}$ and $\rho$  and remark that as ${\rho}$ is minimal, the map $h$ is continuous (see  \cite[Lemma A.4]{KKMj}). For $t\in [0,1]$, write $h_t=(1-t)\id+th$, and observe that for $t\in [0,1)$ this gives a homeomorphism because it is the convex sum of an increasing map with a non-decreasing map. For $t\in (0,1)$, define $\tilde\rho_t:=h_t\tilde \rho h_t^{-1}$. This gives the desired continuous path. Assume now that  ${\rho}$ is cyclic.  This implies that $\tilde\rho$ has a closed discrete orbit $\Lambda\subset \R$. Let $\tilde\rho'\in \mathcal{C}$ be the action obtained by blowing up each point of $\Lambda$ to an interval. Then there is a  semi-conjugacy from $\rho'$ to $\tilde\rho$ which is implemented by a continuous map (obtained by collapsing back the blown-up intervals). There also exists a  semi-conjugacy from $\tilde \rho'$ to ${\rho}$  through a continuous map (obtained by collapsing the complement of the closure of the blown-up intervals). Repeating the argument for the minimal case, we may find continuous paths connecting $\tilde\rho$ first to $\tilde \rho'$ and then to ${\rho}$. \qedhere

\end{proof}

\begin{rem}
	When studying topological properties of $\rep(G, \R)$, such as the space of path-components and the separation properties of semi-conjugacy classes, it is crucial that we restrict the attention to irreducible representations. For example it is not difficult to see that the space of \emph{all} representations of $G$ into $\homeo_0(\R)$ always deformation retracts onto the trivial representation (using a version of the classical ``Alexander trick''),  moreover every  semi-conjugacy class accumulates on the trivial representation (see Mann and Wolff \cite[Proposition 2.13]{MW}).
\end{rem}

We finally recall the fundamental classification of minimal group actions on the line into three types according to their centralizer. 
 Given an action $\rho\colon G\to \homeo_0(\R)$, we denote by $\mathcal{Z}(\rho)$ the centraliser of $\rho(G)$ in $\homeo_0(\R)$. We say that two points $x, y\in \R$ are \emph{proximal} for $\rho$ if there exists a sequence $(g_n)\subset G$ such that $\rho(g_n)(x)$ and $\rho(g_n)(y)$ converge to the same point of $\R$. The action $\rho$ is \emph{proximal} if all pairs of points are proximal, and \emph{locally proximal} if every point is contained in some open interval whose endpoints are proximal.
The following result is proven in Malyutin \cite{Malyutin}.
\begin{thm}\label{t-centraliser}
	Let $G$ be any group and $\rho\in \rep(G;\R)$ a minimal action. Then one of the following holds.
	\begin{enumerate}[label=(\Roman*)]
	\item Either $\rho$ is conjugate to an action by translation, and $\mathcal{Z}(\rho)$ is isomorphic  to $(\R, +)$; or
	\item $\rho$ is locally proximal, but not proximal, $\mathcal{Z}(\rho)$ is conjugate to an infinite cyclic group of translations, and the action of $G$ on the circle $\R/\mathcal{Z}(\rho)\cong \mathbb{S}^1$ is proximal; or
	
	\item $\rho$ is proximal, and  $\mathcal{Z}(\rho)$ is trivial.
	
	\end{enumerate}
	In particular, in all cases, $\mathcal{Z}(\rho)$ is abelian. 
\end{thm}

\section{The Deroin space} \label{sc:Deroin}

\subsection{Definition of the Deroin space} Let $G$ be a discrete group and consider the space $\rep(G;\R)$. Given $s\in\R$, denote by $T_s\in\homeo_0(\R)$ the translation by $s$. We define the \emph{translation flow} $\Psi:\R\times\rep(G;\R)\to \rep(G;\R)$, given by $\Psi^s(\rho)=T_{-s}\circ\rho\circ T_s$. We also have an involution $\hat\iota$ on $\rep(G;\R)$ defined by conjugation by the reflection $x\mapsto -x$. Together with $\Psi$, this gives an action of the group $\mathsf{Isom}(\R)$ of isometries of $\R$ on $\rep(G;\R)$, which is continuous with respect to the compact-open topology. 
We also define the \emph{translation action}  $G\times\rep(G;\R)\to\rep(G;\R)$ by $g. \rho=\Psi^{\rho(g)(0)}(\rho)$.
A proof that this is a well-defined action can be found in Deroin \cite[pp.\ 187--188]{Deroin}.  It is clear that the translation action is continuous and preserves every $\Psi$-orbit.

In other words, one can think of an element $\rho\in \rep(G;\R)$ as an action on the real line with a marker on the point $0$. The translation flow corresponds to moving the marker by a translation, while the translation action of $G$ moves the marker according to the underlying $G$-action. With this point of view in mind, it is useful to introduce the following terminology. 
\begin{dfn}\label{dfn.pointedsemi}
	Consider irreducible actions $\rho_1,\rho_2\in\rep(G;\R)$.
	\begin{enumerate}
		\item We say that $\rho_1$ and $\rho_2$ are \emph{pointed conjugate} if there exists an orientation-preserving conjugacy  $h\colon \R\to \R$ between  $\rho_1$ and $\rho_2$ such that $h(0)=0$.
		\item We say that $\rho_1$ and $\rho_2$ are \emph{pointed semi-conjugate} if there exist $\rho_\ast\in\rep(G;\R)$, either minimal or cyclic, and semi-conjugacies $h_i$ between $\rho_i$ and $\rho_\ast$ for $i\in \{1,2\}$, such that $h_1(0)=h_2(0)$. 
	\end{enumerate}
\end{dfn}
\begin{rem}\label{rem-pointedsemicon}
	Notice that pointed semi-conjugacy is an equivalence relation, which is (strictly) finer than semi-conjugacy. Also, notice that if $\rho_1$ and $\rho_2$ are minimal or cyclic, then they are pointed semi-conjugate if and only if they are pointed conjugate. When $\rho_1$ and $\rho_2$ are cyclic, this is clear, and when they are minimal it follows from the fact that a semi-conjugacy between minimal actions is automatically a conjugacy.\end{rem}

Deroin showed in \cite{Deroin} that if $G$ is finitely generated, then every representation $\rho\in \rep(G; \R)$ is conjugate to one which belongs to some compact $\Psi$-invariant subset of $\rep(G; \R)$. A more universal version of this construction follows from the results in \cite{DKNP} and allows to find a single such set which intersects all semi-conjugacy classes, see \cite[Chapter 14]{BMRT}. The following definition captures the main properties of this construction which will be important for our purposes.
\begin{dfn}\label{d-Deroin}
	Let $G$ be a finitely generated group. We say that a subset $\mathcal{D}\subset\rep(G;\R)$ is a \emph{Deroin space} for $G$ if it satisfies:
	\begin{enumerate}
	\item $\mathcal{D}$ is compact, and invariant under the translation flow $\Psi$; 
	\item every $\rho\in\mathcal{D}$ is either minimal  or cyclic;
	\item every $\rho\in \rep(G;\R)$ is pointed semi-conjugate to a unique element of $\mathcal{D}$. 
	\end{enumerate}
\end{dfn}
\begin{rem}
	As a consequence of the definition,  any Deroin space is $G$-invariant for the translation action of $G$, for the latter preserves orbits of the translations flow. Moreover the $G$-orbit of every  $\rho\in \mathcal{D}$ is dense in its $\Psi$-orbit (when $\rho$ is cyclic this is trivial since the translation flow fixes $\rho$, else it follows from the minimality of $\rho$ by the definition of translation action).
	\end{rem}

Given a Deroin space $\mathcal{D}\subset\rep(G;\R)$, we can define a map \[r_\mathcal{D}\colon \rep(G;\R)\to\mathcal{D}\] that associates to each irreducible representation its unique representative up to pointed semi-conjugacy in $\mathcal{D}$. 
\begin{prop}\label{l-rprop}
	Let $\mathcal D$ be a Deroin space for a finitely generated group $G$.
	The map $r_{\mathcal{D}}$ is a $G$-equivariant continuous retraction, which preserves pointed-semi-conjugacy classes. Two actions $\rho_1, \rho_2\in \rep(G;\R)$ are semi-conjugate if and only if $r_\mathcal{D}(\rho_1)$ and $r_\mathcal{D}(\rho_2)$ belong to the same $\Psi$-orbit in $\mathcal{D}$. \end{prop}
\begin{proof}
	The proof that the map $r_\mathcal{D}$ is a continuous retraction is given  in \cite[Theorem 14.2.1]{BMRT}. (This result is stated there for a specific realisation of Deroin space,  namely the space of normalised $\mu$-harmonic actions, but its proof only relies on the abstract properties  in Definition \ref{d-Deroin}). It remains to prove $G$-equivariance of $r_{\mathcal{D}}$. For this, take $\rho_0\in\rep(G;\R)$ and let $\rho_1:=r_{\mathcal{D}}(\rho_0)$ be its projection to $\mathcal{D}$. Then, there exist a minimal or cyclic model $\rho_\ast\in\rep(G;\R)$ and semi-conjugacies $h_i$ between $\rho_i$ and $\rho_\ast$ for $i\in \{0,1\}$, so that $h_0(0)=h_1(0)$. Thus we have \[h_0(\rho_0(g)(0))=\rho_\ast(g)(h_0(0))=\rho_\ast(g)(h_1(0))=h_1(\rho_1(g)(0)),\]
	which implies that $g.\rho_0$ and $g.\rho_1$ are pointed semi-conjugate. Since Deroin spaces intersect each pointed-semi-conjugacy class in a single point, we have that $r_{\mathcal{D}}(g.\rho_0)=g.\rho_1$. 
	To show the last sentence, note that if $r_\mathcal{D}(\rho_1)$ and $r_\mathcal{D}(\rho_2)$ are actions in $\mathcal{D}$ that belong to the same orbit of the translation flow, then they are evidently conjugate, so that $\rho_1$ and $\rho_2$ are semi-conjugate. Conversely if $\rho_1$ and $\rho_2$ are semi-conjugate, then so are $r_\mathcal{D}(\rho_1)$ and $r_\mathcal{D}(\rho_2)$, and since actions in $\mathcal{D}$ are minimal or cyclic, this implies that they are actually conjugate. So some conjugate of $r_\mathcal{D}(\rho_1)$ by a translation is actually pointed conjugate to $r_\mathcal{D}(\rho_2)$, and by uniqueness this is possible only if  it is equal, so $r_\mathcal{D}(\rho_1)$ and $r_\mathcal{D}(\rho_2)$ belong to the same $\Psi$-orbit.
\end{proof}

\begin{dfn}
Let $(Y_1, \Phi_1), (Y_2, \Phi_2)$ be spaces with a flow. A \emph{flow equivalence} is a homeomorphism $\mathfrak{h}\colon Y_1\to Y_2$ such that for every $y\in Y_1$, there exists an orientation-preserving homeomorphism $\tau\colon \R \to \R$ fixing $0$ such that $\mathfrak{h}(\Phi_1^s(y))=\Phi_2^{\tau(s)}(\mathfrak{h}(y))$ for every $s\in \R$.
\end{dfn}
\begin{thm}[Existence and uniqueness \cite{DKNP, BMRT}]\label{t-existenceDeroin}
	Every finitely generated group has a Deroin space. 

	Moreover, such a space is unique in the following sense: if $\mathcal{D}_1,\mathcal{D}_2\subset\rep(G;\R)$ are Deroin spaces of $G$, there exists a $G$-equivariant flow equivalence $\mathfrak{h}:(\mathcal{D}_1, \Psi|_{\mathcal{D}_1})\to(\mathcal{D}_2, \Psi|_{\mathcal{D}_2})$.
\end{thm}
\begin{proof}
	Existence follows from the results of Deroin, Kleptsyn, Navas, Parwani \cite{DKNP}; see \cite[Theorems 14.1.6 and 14.2.1]{BMRT}. For uniqueness, consider the map $\mathfrak{h}:\mathcal{D}_1\to\mathcal{D}_2$ defined as the restriction of $r_{\mathcal{D}_2}$ to the Deroin space $\mathcal{D}_1$. By Proposition \ref{l-rprop} this map is $G$-equivariant, continuous, and preserves pointed-semi-conjugacy classes. Since each pointed-semi-conjugacy class has exactly one representative in each Deroin space, the map $\mathfrak{h}$ is bijective. It remains to check that $\mathfrak{h}$ is a flow equivalence. For this, consider $\rho_1\in\mathcal{D}_1$ and $\rho_2:=\mathfrak{h}(\rho_1)$. Since $\rho_1$ and $\rho_2$ are pointed semi-conjugate, and they are minimal or cyclic, by Remark \ref{rem-pointedsemicon}, there exists an orientation preserving homeomorphism $\tau:\R\to\R$ which conjugates $\rho_1$ and $\rho_2$, and sends $0$ to itself. Thus, since $\mathfrak{h}$ preserves pointed-semi-conjugacy classes, it must hold that $\mathfrak{h}(\Psi^s(\rho_1))=\Psi^{\tau(s)}(\rho_2)$. 
\end{proof} 
In view of the previous theorem, we can abuse terminology and speak  of \emph{the} Deroin space of $G$ as a  compact space with a flow $(\mathcal D, \Psi)$ and a $G$-action preserving $\Psi$-orbits, which is well defined up to $G$-equivariant flow equivalence. Note that the space $\mathcal{D}$ can be more abstractly described, up to homeomorphism, as the quotient of $\rep(G; \R)$ by the pointed-semi-conjugacy relation. It can also be identified with a quotient of the space of left-invariant preorders on $G$, see \cite[Theorem 14.3.16]{BMRT}.

 \subsection{First properties} We conclude with some basic properties of the Deroin space, which follow easily from the definition and Proposition \ref{l-rprop}.

\begin{dfn}
A flow  $(Y, \Phi)$ is \emph{freely reversible} if there is a fixed-point-free involution (that is, a homeomorphism of order 2) $\iota\colon Y\to Y$ which establishes a flow equivalence between the flow $(Y, \Phi)$ and its time-reverse, denoted $(Y, \Phi^{-1})$. 
\end{dfn}

If $G$ is a finitely generated group, let $\hat\iota \colon \rep(G; \R)\to \rep(G; \R)$ be the conjugation map by the reflection $x\to -x$.  If $\mathcal{D}$ is a Deroin space, then $\hat{\iota}$ descends to a map $\iota\colon \mathcal{D} \to \mathcal{D}$, given by $\iota=r_\mathcal{D}\circ \hat{\iota}$. 

\begin{prop} \label{p-Deroin-freely-reversible}
Let $\mathcal{D}\subset \rep(G; \R)$ be a Deroin space for $G$. Then the map $\iota\colon \mathcal{D}\to \mathcal{D}$ is a fixed-point-free involution which establishes a $G$-equivariant flow-equivalence between $(\mathcal{D}, \Psi)$ and $(\mathcal{D}, \Psi^{-1})$. In particular, $(\mathcal{D}, \Psi)$ is freely reversible.  
\end{prop}

\begin{proof}
The fact that $\iota$ is a $G$-equivariant flow equivalence between $(\mathcal{D}, \Psi)$ and $(\mathcal{D}, \Psi^{-1})$ follows from Proposition \ref{l-rprop}. To observe that it is fixed-point free, assume by contradiction that $\iota(\rho)=\rho$ for some $\rho\in \mathcal{D}$. This implies that every element $\rho(g)$  commutes with an orientation-reversing homeomorphism $h\colon \R\to \R$ fixing $0$. But $0$ is the unique fixed point of $h$, so it must also be fixed by $\rho(G)$, contradicting that $\rho$ is irreducible. \qedhere
\end{proof}

The next proposition gives a precise dictionary with the statement of Theorem \ref{t-centraliser}. It will not be used elsewhere in this paper, but we include it for completeness.
\begin{prop}[Structure of orbits and $G$-invariant measures] \label{p-fixed-points}
Let $\mathcal{D}$ be a Deroin space for the finitely generated group $G$ with flow $\Psi$. Fix $\rho\in \mathcal D$, and write $\ell\subset \mathcal D$ for its $\Psi$-orbit. Then we have the following alternative. 
\begin{enumerate}[label=(\Roman*)]
	
	\item \label{i-Deroin-typeI} The orbit $\ell$ is a point if and only if $\rho$ is semi-conjugate to an action by translation.  In particular, the set of fixed points of $\Psi$  is  homeomorphic to the sphere $\mathbb{S}^{b_1(G)-1}$, where $b_1(G)=\mathsf{rk}(H^1(G,\Z))$ (with $\mathbb{S}^{-1}=\varnothing$). 

\item \label{i-closed-orbit} The orbit $\ell$ is a topological circle if and only if $\rho$ is locally proximal, but not proximal.

\item \label{i-non-closed} The orbit $\ell$ is not closed if and only if $\rho$ is proximal.

\end{enumerate}
Moreover, every $G$-invariant probability measure on $\mathcal{D}$ must be supported on the set of fixed points of $\Psi$. 
\end{prop}
\begin{proof}
By definition,   $\rho$ is fixed by $\Psi$ if and only $\rho(G)$ commutes with all translations, which happens if and only $\rho$ is an action by translations. Conversely, assume that $\rho\in \mathcal{D}$  is semi-conjugate to an action by translations, and thus actually conjugate to it (since all actions in $\mathcal{D}$ are minimal or cyclic). Then all $\Psi$-translates of $\rho$  are pointed conjugate to it. Since $\mathcal{D}$ contains a unique representative of each pointed-semi-conjugacy class, we have that $\rho$ is actually fixed by $\Psi$. This shows the first part of \ref{i-Deroin-typeI}, and a similar argument gives \ref{i-closed-orbit}, and thus \ref{i-non-closed}, using Theorem \ref{t-centraliser}. Now, the set of fixed points is homeomorphic to the space of non-trivial homomorphisms of $G$ to $(\R, +)$ up to positive constant, which is either empty or homeomorphic to $\mathbb{S}^n$, where $n$ is the torsion-free rank of the abelianisation of $G$ minus 1.

To show the last statement, let $\mu$ be a $G$-invariant probability measure on $\mathcal{D}$. Then its support $\supp \mu$ is a closed $G$-invariant subset, and since the $G$-orbit of every $\rho\in \mathcal{D}$ is  dense in its $\Psi$-orbit,  $\supp \mu$ is also $\Psi$-invariant. Assume  by contradiction that $\rho\in \supp \mu$ is not a fixed point, then it is a locally proximal action by Theorem \ref{t-centraliser}. Hence we can find $g\in G$  and an interval $I=(-a, a)$ and $0<\varepsilon<a$ such that $\rho(g)(\overline{I})\subset J:=(-a+\varepsilon, a-\varepsilon)$. Then $U=\{\rho' : \rho'(\overline{I})\subset J\}$ is an open neighbourhood of $\rho$ in  $\mathcal{D}$. Let $U_I=\bigcup_{t\in I} \Psi^t(U)$, and $U_J$ be defined similarly. Then by definition of the translation action we have $g^n. U_I\subset U_J$ for every $n\ge 0$, and thus $\mu(U_I \setminus \overline{U_J})=0$. Hence $U_I\setminus \overline{U_J}$ is an open subset avoiding the support of $\mu$. However, it contains points in the orbit of $\rho$, namely $\Phi^t(\rho)$ for every $t\in I\setminus \overline{J}$. This contradicts the $\Psi$-invariance of $\supp \mu$. \qedhere

\end{proof}

\begin{rem}
The last statement in Proposition \ref{p-fixed-points} implies the theorem of Witte Morris \cite{Witte}, which states that a finitely generated amenable group has a non-trivial representation to $\homeo_0(\R)$ if and only if it has a non-trivial homomorphism to  $(\R, +)$. This proof is  close  to the one given by Deroin \cite{Deroin} using compact $\Psi$-invariant subsets of $\rep(G; \R)$, but avoids the machinery of disintegration of measures along $\Psi$-orbits.
\end{rem}
 
 \begin{rem}
 	From Proposition \ref{p-fixed-points}, one can also deduce that a finitely generated group has the property that its Deroin space contains no closed $\Psi$-orbit if and only if any action on the circle has a fixed point. The groups $\Tsf(\varphi)$ and $\Tsf(\varphi,\sigma)$ (for minimal subshifts $\varphi$) were the first known groups displaying the latter property \cite{MB-T,MR4177294}.
 \end{rem}
 
 Finally we record the following straightforward consequence of Propositions \ref{l-rprop} and \ref{p-path-connected-semiconjugacy}, which further clarifies the relevance of Deroin space from the perspective described in the introduction.
\begin{prop} \label{p-deroin-correspondence}
Let $(\mathcal{D}, \Psi)$ be a Deroin space for a finitely generated group $G$ and $r_\mathcal{D}\colon \rep(G; \R)\to \mathcal{D}$ the associated retraction. 
\begin{enumerate}

\item The open semi-conjugacy-saturated subsets of $\rep(G; \R)$ are precisely the preimages under $r_\mathcal{D}$ of the open $\Psi$-invariant subsets of $\mathcal{D}$. In particular a representation $\rho\in \rep(G; \R)$ is rigid if and only if $r_\mathcal{D}(\rho)$ belongs to an open $\Psi$-orbit.

\item The connected and the path-connected components of $\rep(G; \R)$ coincide with their corresponding preimages of those of $\mathcal{D}$.

\end{enumerate}

\end{prop}

\section{Finitely generated groups of homeomorphisms of suspension spaces}

\subsection{Groups of homeomorphisms of suspension spaces}
Let $X$ be a topological space, and $\varphi\colon X\to X$  a homeomorphism. Recall that we denote by $Y$ the suspension space of $(X, \varphi)$, namely $Y=(X\times \R)/\Z$, where $\Z$ acts on $X\times \R$ by $n\cdot (x, t)=(\varphi^n(x), t-n)$. We let $\pi_Y\colon X\times \R\to Y$ be the quotient projection, and denote by $\Phi$ the suspension flow on $Y$.

 Assume that $(X, \varphi)$ is reversible, with associated involution $\sigma\colon X\to X$ such that $\sigma\varphi\sigma=\varphi^{-1}$. The pair $(\varphi, \sigma)$  then determines an action of the infinite dihedral group $D_\infty$ on $X$, given by $(n, j)\cdot x=\varphi^n \sigma^j(x)$.  Here we identify $D_\infty$ with the semi-direct product $\Z \rtimes (\Z/2\Z)$, and write accordingly its elements as pairs $(n, j)$ with $n\in \Z$ and $j\in \Z/2\Z$. 
We can then consider a smaller suspension space (a ``mapping Klein bottle''). To see this, consider the standard isometric action of $D_\infty$ on $\R$ given by $(n, j)\cdot t= (-1)^jx-n$. Then the quotient 
$Z:=(X\times \R)/D_\infty$
with respect to the associated diagonal action is called the \emph{dihedral suspension} of $(\varphi, \sigma)$. We will denote by $\pi_Z\colon X\times \R \to Z$ the quotient projection. The space $Z$ is a 2-to-1 quotient of the suspension $Y$ of $\varphi$. More precisely, we have $Z= Y/\langle\hat{\sigma}\rangle$, where $\hat\sigma:Y\to Y$ is the involution given by 
\[\hat{\sigma} \colon\, \pi_Y(x, t) \mapsto \pi_Y(\sigma(x), -t).\]
Writing $p\colon Y\to Z$ for the quotient map, we have a commutative diagram:
\begin{center}
\begin{tikzcd}
X\times \R \arrow[rd, "\pi_Z"] \arrow[r, "\pi_Y"] & Y \arrow[d, "p"] \\
& Z=Y/\langle \hat{\sigma}\rangle
\end{tikzcd}
\end{center}
Note that since the involution $\sigma$ normalizes the cyclic group $\langle \varphi \rangle$, it must necessarily send $\varphi$-orbits to $\varphi$-orbits.  The next elementary result clarifies the assumption that we need to put on the involution $\sigma$.
\begin{lem} \label{l-flipping-involution}
	Let $(X,\varphi, \sigma)$ be as above. Then, the following are equivalent:
	\begin{enumerate}
	\item \label{i-sigma-1} $\sigma$ does not preserve any $\varphi$-orbit;
	\item \label{i-sigma-2} $\sigma(x)\notin  \{x, \varphi(x)\}$ for every $x\in X$;
	\item \label{i-sigma-3} every element of order 2 in $D_\infty$ acts on $X$ without fixed points;
	\item \label{i-sigma-4} the map $\hat{\sigma} \colon Y\to Y$ has no fixed point;
	\item \label{i-sigma-5} the diagonal action of $D_\infty$ on $X\times \R$ is free.
	\end{enumerate}
\end{lem}
\begin{proof}
	It is clear that \eqref{i-sigma-1}$\Rightarrow$\eqref{i-sigma-2}. Every element  of order 2 in $D_\infty$ is conjugate to either $\sigma$ or $\sigma\varphi$, so \eqref{i-sigma-2}$\Rightarrow$ \eqref{i-sigma-3}.  The point $\pi_Y(x, t)$ is fixed by $\hat{\sigma}$ if and only if there exists $n\in \Z$ such that $\varphi^{n}\sigma$ fixes $x$ and $t=-n/2$, so  \eqref{i-sigma-3}$\Leftrightarrow$\eqref{i-sigma-4}. Since elements of infinite order already act without fixed points on $\R$, we have \eqref{i-sigma-3}$\Rightarrow$\eqref{i-sigma-5}. Finally if \eqref{i-sigma-1} does not hold, then there exist $x\in X$ and $n\in \Z$ such that $\varphi^{-n}\sigma(x)=x$, and thus the element $\gamma=(-n, 1)\in D_\infty$ fixes the point $(x, -n/2)$, so that \eqref{i-sigma-5}$\Rightarrow$\eqref{i-sigma-1}. \qedhere
\end{proof}

From now on, we will work in the following setting.

\begin{standing-assumption}\label{standing}
	We let $X$ be a totally disconnected, metrisable, compact space, and $(\varphi, \sigma)$ homeomorphisms of $X$ such that $\sigma^2=\id$ and $\sigma \varphi \sigma=\varphi^{-1}$. We further assume that: 
	\begin{itemize}
	\item the action of $\Z$ on $X$ determined by $\varphi$ is topologically free, that is, the subset of points which are not periodic for $\varphi$ is dense in $X$;
	\item the homeomorphism $\sigma$ satisfies the equivalent conditions in Lemma \ref{l-flipping-involution}.
	\end{itemize}
	As above, we always denote by $Y$ the suspension of $\varphi$, with suspension flow $\Phi$, and by $Z$ the dihedral suspension of $(\varphi, \sigma)$. 
\end{standing-assumption}
\begin{ex} \label{e-doubling}
As a basic example, start with any homeomorphism $\varphi_0\colon X_0\to X_0$ of a compact totally disconnected metrisable space. Let $X$ be the disjoint union of two copies of $X$, and define $\varphi$ to be equal to $\varphi_0$ on one copy and to $\varphi_0^{-1}$ on the other. Finally let $\sigma$ exchange the two copies of $X_0$. This class of examples is somewhat degenerate because $\varphi$ preserves a partition into two clopen sets (in particular it can never be topologically transitive).

\end{ex}

\begin{ex}
Let $\varphi$ be an irrational rotation on $\R/\Z$, and $\sigma$ the reflection $x\mapsto -x$. Then the pair $(\varphi, \sigma)$ defines an action of $D_\infty$ on $\R/\Z$. The fixed points of all involutions  in $D_\infty$ belong to exactly 4 orbits (as there are 2 conjugacy classes of reflections and each has two fixed points). Blow up these orbits to obtain a minimal invariant Cantor set $X$. Then $(X, \varphi, \sigma)$ satisfies Assumption \ref{standing}. 
\end{ex}
\begin{ex}\label{ex.reversible}
	The following source of examples is inspired from Hyde and Lodha \cite{HL}, see \cite[Remark 4.19]{confined}. Take  a finite alphabet $A$ of even cardinality $|A|\ge 4$, and assume that $A$ is endowed with a map $a\mapsto a^{-1}$ which is an involution without fixed points. Let $\varphi$ be the shift map on $A^\Z$, and consider also the map $\sigma((a_n)_n)=(a_{-(n+1)}^{-1})_n$ (namely the formal inverse of a bi-infinite word in $A$). It is easy to verify that $\sigma^2=\id$ and $\sigma \varphi \sigma=\varphi^{-1}$, so $(\varphi,\sigma)$ determines an action of $D_\infty$ on $\Omega$. This action does not satisfy the conditions in Assumption \ref{standing}, as $\sigma$ has fixed points.  To avoid this problem, we can consider the subshift of finite type $X_{\text{red}}=\{(a_n)_n\in A^\Z:a_{n+1}\neq a^{-1}_n\}$ of \emph{reduced}  words in $A$, which is closed and $D_\infty$-invariant. Then,  Assumption \ref{standing} is satisfied by $(X_{\text{red}},\varphi,\sigma)$. One can then obtain many more examples as $D_\infty$-invariant closed subsets $X\subset X_{\text{red}}$. 
\end{ex}

The set $Z$ is naturally partitioned into subsets of the form $\ell=\pi_Z(\{x\} \times \R)$, that we call the \emph{leaves} of $Z$. As $X$ is totally disconnected, these coincide with the path components of $Z$. By Lemma \ref{l-flipping-involution}, the map $p\colon Y\to Z$ is a local homeomorphism, and is injective in restriction to each $\Phi$-orbit. 
Take now  a clopen subset $C\subset X$, and an open interval $I\subset \R$. If the restriction of $\pi_Y$ to $C\times I$ is injective, we denote by  $Y_{C, I}$ its image. Then the map 
\[\pi_Y\colon C\times I \to Y_{C, I}\subset Y\] 
is a homeomorphism, and is called a  \emph{chart} of $Y$.
Similarly, if the restriction of $\pi_Z$ to $C\times I$ is injective, we denote its image by $Z_{C, I}$ and call the map
\[\pi_Z\colon C\times I \to Z_{C, I}\subset Z\] 
a chart of $Z$. Most of the time, we will  just refer to the sets $Y_{C, I}$ and $Z_{C, I}$ as charts, with an implicit identification with $C\times I$. 
Since the actions of $\Z$ and  $D_\infty$ on $X\times \R$ are free and properly discontinuous, both maps $\pi_Y$ and $\pi_Z$ are local homeomorphisms, and thus the spaces $Y$ and $Z$ can be covered by charts.

We now consider the group $\mathsf{H}_0(\varphi)$ (respectively, $\mathsf{H}_0(\varphi, \sigma)$)  of all homeomorphisms of $Y$  (respectively, $Z$) isotopic to the identity.
Note that this condition forces all elements of  $\mathsf{H}_0(\varphi)$ (respectively, $\mathsf{H}_0(\varphi, \sigma)$) to preserve each $\Phi$-orbit in $Y$ (respectively, each leaf of $Z$), since these are exactly the path components of the corresponding space.
By \cite[Theorem 3.1]{BCE}, the group $\mathsf{H}_0(\varphi)$ is exactly the group of homeomorphisms $h\colon Y\to Y$ for which there exists a continuous function $\tau_h\colon Y\to \R$ such that 
\begin{equation}\label{e-cocycle}
	h(y)= \Phi^{\tau_h(y)}(y).
\end{equation}
It follows from Assumption \ref{standing} that such a function $\tau_h$ is necessarily unique. Indeed, the value $\tau_h(y)$ is uniquely determined by $h(y)$ provided $y$ does not belong to a $\Phi$-periodic orbit. By our assumption such points are dense, and thus $\tau_h$ is uniquely determined everywhere. The function $\tau_h$ will be called the \emph{translation cocycle} of $h$.

\begin{lem} \label{l-lifting-H0}
	Let $(X,\varphi,\sigma)$ be as in Assumption \ref{standing}. Then, the action of $\mathsf{H}_0(\varphi, \sigma)$ on $Z$ lifts through the map $p\colon Y \to Z$ to a unique action on $Y$ which preserves every $\Phi$-orbit. This action identifies $\mathsf{H}_0(\varphi, \sigma)$ with the centralizer of $\hat{\sigma}$ in $\mathsf{H}_0(\varphi)$, which  consists of all elements $h$ such that $\tau_h \circ \hat{\sigma}=-\tau_h$. 
\end{lem}
\begin{proof}
	By Lemma \ref{l-flipping-involution}, the map $p$ sends each $\Phi$-orbit bijectively onto a leaf $\ell$ of $Z$.  Hence, at least at a set-theoretic level, the action of $\mathsf{H}_0(\varphi, \sigma)$  on $Z$ can be leafwise lifted   to  a unique action on $Y$ which preserves every $\Phi$-orbit, and such that the map $p$ is equivariant. To check that this action is continuous, fix a finite cover by charts $Z=\bigcup_{i=1}^r Z_{C_i, I_i}$; each chart $Z_{C_i, I_i}$ lifts to a pair of disjoint charts in $Y$, mapped homeomorphically onto it by $p$. Take $\varepsilon>0$ such that every arc of length at most $\varepsilon$ in a leaf of $Z$ is contained in one of  the charts  $Z_{C_i, I_i}$. If  $g\in \mathsf{H}_0(\varphi, \sigma)$ is an element that displaces any point of $Z$ by a distance less than $\varepsilon$ on the corresponding leaf, then one can check on charts that the lift of $g$ to $Y$ defined above is continuous. But such elements generate $\mathsf{H}_0(\varphi, \sigma)$, as all elements are isotopic to the identity. This implies that the lift of any element is continuous. By the same reasoning, every isotopy of an element $g\in\mathsf{H}_0(\varphi, \sigma)$  to the identity in the group of homeomorphism of $Z$ lifts to an isotopy in the group of homeomorphisms of $Y$. Therefore $\mathsf{H}_0(\varphi, \sigma)$ lifts to a subgroup of $\mathsf{H}_0(\varphi)$. Moreover,  its action clearly commutes with $\hat{\sigma}$. 
	Conversely, assume that $h\in \mathsf{H}_0(\varphi)$ commutes with $\hat{\sigma}$. It is easily checked that this is equivalent to  $\tau_h\circ \hat{\sigma}=-\tau_h$. Choose the isotopy $h_s$ of $h$ to the identity given by
	$h_s(y)=\Phi^{(1-s)\tau_h(y)}(y)$.
	Then each element $h_s$ commutes with $\hat{\sigma}$ and thus descends to a homeomorphism of $Z$. This shows that $h$ defines indeed an element of $\mathsf{H}_0(\varphi, \sigma)$. \qedhere
\end{proof}

\subsection{Perfect germ extensions of (solenoidal) Thompson's groups} \label{s-Thompson}
Here we consider variations of the classical Thompson's groups. We will use some of the standard properties of these groups (such as finite generation) that the reader can find explained in standard references, such as the text by Cannon, Floyd, and Parry \cite{CFP}.

Let $I,J\subset \R$ be open dyadic intervals. We say that a homeomorphism $f\colon I\to J$ is \emph{dyadic piecewise linear (PL)} if there exists a finite subset $\Sigma\subset I$ so that:
\begin{itemize}
	\item $\Sigma\subset\Z[\frac{1}{2}]$, and
	\item $f$ is locally dyadic affine outside $\Sigma$, that is, on each connected component of $I\setminus\Sigma$, the map is defined as $x\mapsto 2^nx+b$ for some $n\in\Z$ and $b\in\Z[\frac{1}{2}]$.
\end{itemize} 
We are interested in the family of maps which are locally dyadic PL outside a finite subset, and satisfy an extra local condition. To define this properly, given a dyadic rational, let $h_{x_0}\colon \R\to\R$ be the map defined by $h_{x_0}(x)=2(x-x_0)+x_0$. 

\begin{dfn}
	Say that a map $f\colon I\to J$ is \emph{of type} $\nice$ if there exists a finite subset $\mathsf{BP}^2(f)\subset I$ such that
	\begin{itemize}
		\item every $x\in I\setminus \mathsf{BP}^2(f)$ has an open neighborhood $U$ so that $f|_{U}$ is dyadic PL, and
		\item every $x_0\in\mathsf{BP}^2(f)$ has a neighborhood $U$ so that $f\circ h_{x_0}=h_{f(x_0)}\circ f$ for $x\in U$.
	\end{itemize}
\end{dfn} 

Given an interval $I\subset \R$, we denote by $\Fsing_I$ the group of homeomorphisms of $\overline I$ of type $\nice$, and $\Fsing:=\Fsing_{(0, 1)}$. These are all isomorphic if $I$ is dyadic. Note that the subgroup of all elements $f\in \Fsing$ such that $\BP^2(f)=\varnothing$ is the standard Thompson's group $F$.
Given a dyadic $x\in [0,1]$, denote by $\mathcal{D}_x^+$ (respectively, $\mathcal{D}_x^-$) the group of right (respectively, left) germs of $\Fsing$ at $x$. Recall that these are defined as the quotient of the stabiliser of $\Fsing$ by the normal subgroup of elements acting trivially on a right (respectively, left) neighbourhood. Write $\tilde{T}\subseteq \homeo_0(\R)$ for the $\Z$-central extension of Thompson's group $T$, obtained by lifting the action of $T$ on the circle. Explicitly, $\tilde{T}$ is the group of all (locally) dyadic PL homeomorphisms of $\R$ commuting with the unit translation $t_1\colon x\mapsto x+1$.

\begin{lem}\label{lem-Ttilde}
	The groups $\mathcal{D}^+_x$ and $\mathcal{D}^-_x$ are isomorphic to $\tilde{T}$ for every $x\in\Z[\frac{1}{2}] \cap [0, 1]$. \end{lem}
\begin{proof}
	Take $x\in (0,1)$ dyadic and notice that the map $h_{x}|_{(-\infty,x)}$  is  conjugate to the translation $t_1$ by a  dyadic PL homeomorphism $f\colon (-\infty, x)\to \R$. This establishes an isomorphism between $\mathcal{D}_{x}^-$ and the group of germs of $\tilde{T}$ at $+\infty$, which is isomorphic to $\tilde{T}$ itself. \end{proof}

\begin{lem}\label{l-fingenperfect}
	The group $\Fsing$ is finitely generated and perfect.
\end{lem}
\begin{proof}
	Fix a dyadic $x\in (0,1)$. Notice that the group of germs at $x$, that we denote by $\mathcal{D}_x$, is isomorphic to the direct product $\mathcal{D}_x^-\times\mathcal{D}_x^+$. Thus, by Lemma \ref{lem-Ttilde}, we have that $\mathcal{D}_x$ is isomorphic to $\tilde{T}\times\tilde{T}$. In particular it is finitely generated, so we can choose a finite subset $S\subset\Fsing$ that generates $\mathcal{D}_x$. Moreover, we can ask every $f\in S$ to satisfy $\mathsf{BP}^2(f)=\{x\}$. Thus, given $g\in\Fsing$, we can find $g_\ast$ in the group $\langle F',S\rangle$ so that $\mathsf{BP}^2(gg_\ast)\cap (0,1)=\emptyset$.  Analogously, we can also consider finite subsets $S_0$ and $S_1$ generating $\mathcal{D}_0^+$ and $\mathcal{D}_1^-$, respectively. Therefore we can write
	\[\Fsing=\langle F',S,S_0,S_1\rangle=\langle F,S,S_0,S_1\rangle.\]
	Since Thompson's group $F$ is finitely generated, we deduce that so is $\Fsing$. We next prove that $\Fsing$ is perfect. For this, notice that, since $\tilde{T}$ is perfect, we can find $S$, $S_0$ and $S_1$ consisting of products of commutators. Finally, since $F'$ is also perfect it follows that $\Fsing$ is perfect. 
\end{proof}

We now define two countable subgroups of the groups $\mathsf{H}_0(\varphi)$ and $\mathsf{H}_0(\varphi, \sigma)$, respectively, which are type-$\nice$ analogues of the group $\Tsf(\varphi)$ defined in \cite{MB-T}. 

\begin{dfn}
	We let $\Tsingsf(\varphi)$ be the group of homeomorphisms $g$ of $Y$ such that for every $y\in Y$, there exist charts $Y_{C, I}, Y_{C, J}$ containing $y$ and $g(y)$, respectively, and a homeomorphism $f\colon I\to J$ of type $\nice$ such that  $g(Y_{C, I})=Y_{C, J}$, and $g|_{Y_{C, I}}$ is given by
	\[g(\pi_Y(x, t))=\pi_Y(x, f(t)) \quad \text{for any }(x, t)\in C\times Y.\]
	
	Similarly, we let  $\Tsingsf(\varphi, \sigma)$ be the group of homeomorphisms $g$ of $Z$ such that for every $y\in Y$, there exist charts $Z_{C, I}, Z_{C, J}$ containing $y$ and $g(y)$, respectively, and a homeomorphism $f\colon I\to J$ of type $\nice$ such that  $g(Z_{C, I})=Z_{C, J}$, and $g|_{Z_{C, I}}$ is given by
	\[g(\pi_Z(x, t))=\pi_Z(x, f(t)) \quad  \text{for any } (x, t)\in C\times Y.\]
\end{dfn}

\begin{lem} 
	We have $\Tsingsf(\varphi)\le \mathsf{H}_0(\varphi)$ and $\Tsingsf(\varphi, \sigma)\le \mathsf{H}_0(\varphi, \sigma)$.  The group $\Tsingsf(\varphi, \sigma)$ coincides with the centralizer of $\hat{\sigma}$ in $\Tsingsf(\varphi)$. 
\end{lem}

\begin{proof}
	For $g\in \Tsingsf(\varphi)$, one can define a continuous function $\tau_g$ as in \eqref{e-cocycle} locally around any point  $y\in Y$, by choosing charts $Y_{C, I}, Y_{C, J}$, and $f\colon I\to J$ as in the definition, and setting $\tau_g(p(x, t))=f(t)-t$. This definition must agree  on overlapping charts,  since $\tau_g(y)$ satisfying \eqref{e-cocycle} is uniquely determined for $y$ in a dense set. Thus $\Tsingsf(\varphi)\le \mathsf{H}_0(\varphi)$. The rest of the proof is similar to that of Lemma \ref{l-lifting-H0}.
\end{proof}
The following shows that the group $\Tsingsf(\varphi, \sigma)$ is finitely generated provided $\varphi\colon X\to X$ is conjugate to a subshift (or equivalently, if it is {expansive}).

\begin{thm}\label{t:finitely_generated}
	Let $(X,\varphi, \sigma)$  be a reversible subshift satisfying Assumption \ref{standing}.  Then, $\Tsingsf(\varphi, \sigma)$ is finitely generated.
\end{thm}

This is the analogue of \cite[Theorem A]{MB-T}. The argument in the proof is essentially the same but needs to be  slightly modified, as it makes use of the existence of elements in $\Tsf(\varphi)$ that move all points in the suspension space $Y$ strictly in the direction of the flow.  This is not possible here as there is no globally well-defined direction on the leaves of $Z$. We provide details in Appendix \ref{appendix}.

\begin{rem}
For any $(X, \varphi)$, the group $\Tsingsf(\varphi)$ can be identified with $\Tsingsf(\varphi_1, \sigma_1)$, where $(X_1, \varphi_1, \sigma_1)$ is obtained from $(X, \varphi)$ through the doubling construction in Example \ref{e-doubling}. Thus any result on the groups $\Tsingsf(\varphi, \sigma)$ can be translated to the groups $\Tsingsf(\varphi)$. 
\end{rem}

\section{Actions on the line of $\Tsingsf(\varphi, \sigma)$}
\subsection{Identifying the Deroin space} \label{s-Deroin-identification}
In this section we analyse actions on the line of the group $G=\Tsingsf(\varphi, \sigma)$, where $(X,\varphi,\sigma)$ is as in Assumption \ref{standing}. Our ultimate goal is to show that, when $G$ is finitely generated, the Deroin space of $G$ can be identified with $(Y, \Phi)$. While finite generation is necessary for the Deroin space to be well-defined, the reader will see that the main result of the section does not require that $G$ be finitely generated: we will prove that any irreducible action of $G$ on the line, admitting a minimal set, comes from $Y$. Thus we do not assume here that $(X, \varphi)$ is conjugate to a subshift.

Recall that given $g\in G$, we use the notation $\tau_g\colon Y\to \R$ for the translation cocycle of $g$. 
Given $y\in Y$, let $\rho_y\in \rep(G;\R)$ be the action given by
\begin{equation} \label{e-action} \rho_y(g)(t)=t+\tau_g(\Phi^t(y)).\end{equation}
When $y$ is not periodic for $\Phi$, this is simply the action of $G$ on the $\Phi$-orbit of $y$, if we identify the orbit with $\R$ using the flow, in such a way that $0$ corresponds to the point $y$. When $y$ is periodic, its orbit is homeomorphic to a circle, and the action $\rho_y$ is a lift of the action of $G$ under the covering map $\R\to \mathbb{S}^1$ determined by the flow and which maps $0$ to $y$. Then we have the following.

\begin{thm}\label{t-classification-actions}
	Let $(X,\varphi,\sigma)$ be as in Assumption \ref{standing}, and write $G:=\Tsingsf(\varphi,\sigma)$. Let $\rho\in \rep(G;\R)$ be a minimal action. Then, there exists $y\in Y$ such that $\rho$ and $\rho_y$ are pointed conjugate.
\end{thm}

Before proving Theorem \ref{t-classification-actions}, let us  deduce Theorem \ref{t-realisation}, which is based on the following consequence. 

\begin{cor}\label{c-mainDeroin}
	Let $(X,\varphi,\sigma)$ be as in Assumption \ref{standing}. Consider the map $q \colon Y\to \rep(G;\R)$ given by $q\colon y\mapsto \rho_y$. Then $q$ is a $G$-equivariant homeomorphism onto its image, which conjugates the flow $\Phi$ to the translation flow $\Psi$ on $q(Y)$. In particular if $(X, \varphi)$ is a subshift, then $(Y, \Phi)$ identifies with a Deroin space for $G$. 
	
\end{cor}
\begin{proof}
	From the explicit formula defining the action $\rho_y$, it is clear that $q$ is continuous. Pick two distinct points $y_1, y_2\in Y$. If $y_1\neq \hat{\sigma}(y_2)$, then also their projections to $Z$ are distinct. It follows that they have distinct stabilisers in $G$. As a consequence, the  actions $\rho_{y_1}$ and $\rho_{y_2}$ are not pointed conjugate, and since they are minimal, they are not pointed semi-conjugate either (see Remark \ref{rem-pointedsemicon}). If $y_1=\hat{\sigma}(y_2)$, then every element $g\in G$ such that $\rho_{y_1}(g)(0)>0$, must satisfy $\rho_{y_2}(g)(0)<0$, and thus also the actions $\rho_{y_i}$ are not pointed semi-conjugate. In particular the map $q$ is injective and thus a homeomorphism onto its image. It is straightforward that $q$  conjugates $\Phi$ to the translation flow and is $G$-equivariant. 
	
	For the last statement, note that after Theorem \ref{t:finitely_generated}, the group $G$ is finitely generated, so that any irreducible action admits a minimal set. Moreover, since the group $G$ is perfect, it does not admit any cyclic action. It thus follows from Theorem \ref{t-classification-actions} that $q(Y)$ contains exactly one representative for each pointed-semi-conjugacy class.\qedhere
\end{proof}

 Corollary \ref{c-mainDeroin} implies Theorem \ref{t-realisation}, using the results of  Bowen and Walters \cite{expansive}. We first recall all the necessary terminology.

\begin{dfn}
Let $(Y, \Phi)$ be a flow on a compact metrisable space, and $d$ a metric on $Y$. Then $\Phi$ is \emph{expansive} if for every $\varepsilon>0$ there exists $\delta>0$ with the property that if $x, y\in Y$ satisfy  $d(\Phi^t(x), \Phi^{s(t)}(y))<\delta$ for every $t\in \R$ and for some continuous map $s\colon \R\to \R$, then $y=\Phi^t(x)$ with $|t|<\varepsilon$. 
\end{dfn}
This property does not depend on the choice of the metric $d$.
Recall also that a space $Y$ has topological dimension $d$ if every open cover $\mathcal{U}$ has a refinement $\mathcal{U'}$ such that every point of $Y$ is contained in at most $d+1$  elements of $\mathcal{U}'$, and $d$ is the least integer with this property. 

\begin{proof}[Proof of Theorem \ref{t-realisation}]

Let $(Y, \Phi)$ be as in Theorem \ref{t-realisation}. It is shown in \cite{expansive} that $\Phi$ admits a cross-section $X\subset Y$ (namely a closed subset which intersects every $\Phi$-orbit in a discrete set of times) and that for every such cross-section,   the first-return map $\varphi$ to $X$ is conjugate to a subshift, so that  $(Y, \Phi)$ is flow-equivalent to the suspension flow of $(X, \varphi)$. The only minor point to address is that we want $(X, \varphi)$ to be reversible and to satisfy Assumption \ref{standing}. Since $(Y, \Phi)$ is freely reversible, there exists  a fixed-point-free involution $\hat{\sigma}\colon Y\to Y$ sending any $\Phi$-orbit to a $\Phi$-orbit in an orientation preserving way.  Let $X_0$ be any cross-section and set $X=X_0\cup \hat{\sigma}(X_0)$. Then $X$ is still a cross-section. Denoting by $\varphi$ the first-return map and $\sigma=\hat{\sigma}|_X$, then $(X, \varphi, \sigma)$ is a reversible subshift. Moreover $\hat{\sigma}$ cannot preserve any $\Phi$-orbit (or it would have a fixed point), so neither does $\sigma$ and by Lemma \ref{l-flipping-involution} we have that $(X, \varphi, \sigma)$ satisfies Assumption \ref{standing}. 
\end{proof}
  The rest of the section is devoted to the proof of Theorem \ref{t-classification-actions}.

\subsection{Finding minimal sets for non-finitely generated groups}\label{ssc:centralizers}
The first step is to show Proposition \ref{p-semicon2} below.  We will make repeated use of the following consequence of Theorem \ref{t-centraliser}.

\begin{cor}\label{c-fixinglambda}
	Let $G$ be a group and $\rho\in \rep (G;\R)$ an irreducible action admitting a perfect minimal set $\Lambda$. Then, every perfect subgroup of $\mathcal{Z}(\rho)$ pointwise fixes $\Lambda$. 
\end{cor}
\begin{proof}
	Since $\Lambda$ is the unique minimal set for $\rho$, it must be preserved by $\mathcal{Z}(\rho)$. Moreover, by collapsing connected components of the complement of $\Lambda$, we can semi-conjugate $\rho$ to a minimal action $\hat\rho\colon H\to\homeo_0(\R)$. Since $\mathcal Z(\rho)$ preserves $\Lambda$, it projects to $\mathcal{Z}(\hat\rho)$. By Theorem \ref{t-centraliser}, we get that any perfect subgroup of $\mathcal{Z}(\rho)$ projects to the trivial group, implying that it pointwise fixes $\Lambda$.
\end{proof}

Given a group $H$ and an integer $n\ge 1$, write $H^n=H_1\times\cdots\times H_n$ for the direct product of $n$ copies of $H$. Denote by $\pi_i\colon H^n\to H_i$ the projection from $H^n$ onto $H_i$.

\begin{lem}\label{l-minimalcomun}
	Let $H$ be a finitely generated perfect group. Consider also an action $\rho\colon H^n\to\homeo_0(\R)$ and a subgroup $K\subseteq H^n$ for which the projections $\pi_i|_{K}$ are surjective for any $i\in \{1,\ldots,n\}$. Then we have the following.
	\begin{enumerate}
		\item\label{i-1} $\fix(\rho|_{K})=\fix(\rho)$.
		\item\label{i-2} If $\rho$ is irreducible, then $\rho|_{K}$ and $\rho$  have the same unique  perfect minimal set.
	\end{enumerate}
\end{lem}
\begin{proof}
	We will first prove \eqref{i-2}. Notice that there is $m\in\{1,\ldots,n\}$ so that $\rho|_{H_m}$ is irreducible, and without loss of generality we can assume $m=n$.  Indeed, if this was not the case, $\rho$ would have a fixed point, contradicting its irreducibility. Moreover, since $H$ is finitely generated, the action $\rho|_{H_n}$ preserves a minimal set that we denote by $\Lambda$.  This set cannot be a single closed orbit, since $H_n$ is perfect, and thus $\Lambda$ is a perfect set. Thus, by Corollary \ref{c-fixinglambda}, we have that the subgroup $H_1\times\cdots\times H_{n-1}$ must  fix $\Lambda$ pointwise. This implies that the action of $\rho$ on $\Lambda$ (that we denote by $\rho_\ast\colon H^n\to \homeo(\Lambda)$) factors through the projection $\pi_n$, that is, $\rho_\ast=\Psi_\ast\circ\pi_n$. On the other hand, by hypothesis $\pi_n|_K$ surjects to $H_n$ and therefore $\rho_\ast|_{K}\colon K\to\homeo(\Lambda)$ has the same image as $\rho_\ast$, implying that it is also a minimal action. This implies that $\Lambda$ is a minimal set of $\rho|_{K}$, as desired. 
	
	We next prove \eqref{i-1}. For this, take  a connected component $J$ of $\R\setminus \fix(\rho)$.  By  part \eqref{i-2} the actions of $H^n$ and $K$ on $J$  have the same minimal set, and in particular $J$ is a connected component of the support of $\rho|_{K}$. Since the component $J$ was arbitrary this implies that the support of $\rho$ is contained in that of $\rho|_K$, which taking complements gives the reverse inclusion $\fix(\rho|_{K})\subseteq\fix(\rho)$. Since the other inclusion is trivial, equality in \eqref{i-1} follows.
\end{proof}

Given a space $X$ and a group $H$, denote by $\Ccal(X, H)$ the group of continuous functions from $X$ to $H$.  Note that this group is not finitely generated in general thus its actions on the line may \textit{a priori} fail to admit a minimal set. 
However, this is the case when $X$ is as in Assumption \ref{standing}, and $H$ is perfect and finitely generated:

\begin{prop}\label{p-semicon2}
	Let $X$ be a metrisable, compact, totally disconnected space, and $H$ a finitely generated perfect group. Then, every action $\rho\in \rep(\Ccal(X, H);\R)$  admits a perfect minimal set. 
\end{prop}

Before proving this result, let us introduce some further notation. Assume $X$ is a compact, totally disconnected space. For a clopen subset $C\subseteq X$, let $H_C\le \mathcal C(X,H)$ be the subgroup (isomorphic to $H$) of functions $\sigma\colon X\to H$ which are constants on $C$ and equal to the identity on $X\setminus C$. Also, if $\mathcal{P}$ is a clopen partition of $X$, write $H^{(\mathcal{P})}:= \langle\bigcup_{C\in \mathcal P} H_C\rangle$, which is naturally isomorphic to the direct product $H^{|\mathcal{P}|}$. 
Recall also that if $\mathcal{P}_1,\mathcal{P}_2$ are clopen partitions of $X$, one says that $\mathcal{P}_2$ is \emph{finer} than $\mathcal{P}_1$ if for every $C\in\mathcal{P}_2$, there exists $D\in\mathcal{P}_1$ so that $C\subseteq D$.

\begin{rem}\label{rem-inclusion}
	When $\mathcal{P}_2$ is finer that $\mathcal{P}_1$, there exists a natural injective morphism $\iota\colon H^{(\mathcal{P}_1)}\to H^{(\mathcal{P}_2)}$, and this morphism satisfies that $\pi\circ\iota$ is surjective for every projection $\pi\colon H^{(\mathcal{P}_2)}\to H_{C}$ given by any clopen subset $C\in \mathcal{P}_2$. This holds in particular when $\mathcal P_1=\{X\}$ is the trivial partition.
\end{rem}

\begin{proof}[Proof of Proposition \ref{p-semicon2}]
	As $X$ is metrisable, we can find an increasingly finer sequence of partitions $\mathcal{P}_n$ of $X$ that separates points. Then we can write $\mathcal C(X,H)$ as an increasing union
	\begin{equation}\label{eq-union}
	\textstyle	\mathcal C(X,H)=\bigcup_{n\in\N}H^{(\mathcal{P}_n)}.
	\end{equation}
	By Remark \ref{rem-inclusion}, for every projection $\pi\colon H^{(\mathcal{P}_n)}\to H_C$ associated with every $C\in \mathcal{P}_n$, the restriction $\pi|_{H_X}\colon H_X\to H_C$ is surjective. Write $\rho_n:=\rho|_{H^{(\mathcal{P}_n)}}$. Thus, applying the first part of Lemma \ref{l-minimalcomun} (with $\rho=\rho_n$ and $K=H_X$), we have the equality
	\begin{equation}\label{eq-equality}
		\fix(\rho_n|_{H_X})=\fix(\rho_n).
	\end{equation}
	Since $\rho_n|_{H_X}=\rho|_{H_X}$, \eqref{eq-equality} gives $\fix(\rho\vert_{H_X})=\fix(\rho_n)$ for every $n\in \N$, whereas \eqref{eq-union} gives $\fix(\rho)=\bigcap_{n\in \N}\fix(\rho_n)$. As $\rho$ is irreducible, we deduce that  $\fix(\rho|_{H_X})=\varnothing$, and thus, by \eqref{eq-equality} again, also $\rho_n$ is irreducible for any $n\in \N$. Therefore, we can apply the second part of Lemma \ref{l-minimalcomun} (applied to $\rho=\rho_n$ and $K=H_X$ again), and obtain that $\rho_n$ has a unique perfect minimal set $\Lambda$, which does not depend on $n\in \N$, so, by \eqref{eq-union}, $\Lambda$ is a minimal set for $\rho$.
\end{proof}

\begin{rem}
We will apply Proposition \ref{p-semicon2} to the group $H=\Fsing$, the perfect germ extension of Thompson's group $F$ introduced in \S \ref{s-Thompson}. We note that if instead we take $H$ to be the usual Thompson's group $F$, then the conclusion of Proposition \ref{p-semicon2} is far from true. Indeed the abelianisation of the group $\mathcal{C}(X, F)$  is  free abelian of infinite rank; one can use this to construct a wealth of faithful actions of $\mathcal{C}(X, F)$ without a minimal set, by defining orders on $\mathcal{C}(X, F)$ through lexicographic constructions similar to the one in \cite[\S 6.1]{BMRT}. This is the reason why we consider a germ extension of PL homeomorphisms to define the groups $\Tsing(\varphi, \sigma)$.
\end{rem}

\subsection{Building an equivariant map to a leaf}
The key step in proving Theorem \ref{t-classification-actions} is the following result, which says that any minimal action of $G$ comes from an action on a leaf of $Z$.

\begin{prop}\label{t-hmap}
	With assumptions as in Theorem \ref{t-classification-actions}, there exists a continuous map $h\colon \R \to Z$ which is $G$-equivariant (with respect to the $\rho$-action on $\R$ and the natural action on $Z$), locally injective, and takes values in a single leaf.
\end{prop}

The main difference from the statement of Theorem \ref{t-classification-actions} is that the map is to a leaf of $Z$ (instead of $Y$), and moreover the leaf can be closed. In the latter case, we want to prove that the minimal action $\rho$ is conjugate to the lift of the action on the leaf to its universal cover. For this, we introduce the group of circle homeomorphisms $\Tsing$, defined to be the perfect germ extension of Thompson's group $T$ acting on the circle, namely the group of all homeomorphisms of $\mathbb{S}^1$ of type $\nice$. We further let $\tilde{\Tsing}$ denote its lift to $\R$, namely the group of all homeomorphisms of $\R$ which are locally of type $\nice$ and commute with integer translations. Thus,  $\tilde{\Tsing}$ has an infinite cyclic center $\mathcal{Z}(\tilde{\Tsing})$ and $\tilde{\Tsing}/\mathcal{Z}(\tilde{\Tsing})=\Tsing$. We will need the following lemma. 

\begin{lem}\label{l-tilde-T}
	Let $f\colon \tilde{\Tsing} \to \tilde{\Tsing}$ be a group homomorphism such that $f(\mathcal{Z}(\tilde{\Tsing}))\subset \mathcal{Z}(\tilde{\Tsing})$ and such that $f$ induces in the quotient the identity map $\Tsing\to \Tsing$. Then $f$ is the identity. 
\end{lem}

\begin{proof}
	The analogous statement for the usual Thompson's group $T$ this is proven at the end of the proof of \cite[Theorem 8.7]{MB-T}. The same proof works for $\Tsing$. \qedhere
\end{proof}

We will also use the fact that the group $G$ is perfect. For this, given an element $g\in G$, we call the set of points of $Z$ (respectively, $Y$) that are moved by $g$ the \emph{support} of $g$, and denote it by  $\supp_Z(g)$ (respectively, $\supp_Y(g)$). Note that the support is an open set. 
For an open subset $U\subset Z$ we let $G_U$ be the subgroup of elements with support contained in $U$; if $U=Z_{C, I}$ is a chart, we also denote $G_U$ by $G_{C, I}$, and use analogous notation for the action on $Y$. 
In order to prove that $G$ is perfect, we prove that it is generated by the perfect subgroups $G_{C,I}$ (see the proof of Lemma \ref{l-minimal-CIxi} below), provided that we take charts $Z_{C,I}$ covering $Z$. More precisely, we have the following fundamental property, which is the analogue of \cite[Lemma 4.8]{MB-T}. The proof is discussed in Appendix \ref{appendix}. 

\begin{prop}[Fragmentation lemma] \label{p-fragmentation}
	Let $(X,\varphi, \sigma)$  be as in Assumption \ref{standing}, and set $G=\Tsingsf(\varphi, \sigma)$. For any element $g\in G$ and open cover $\mathcal{U}$ of $\overline{\supp_Z(g)}$, we have
	$g\in \left \langle \bigcup_{U\in \mathcal U}G_U\right \rangle$.
	
	In particular, the group $G$ is perfect.
\end{prop}

Let us see now how to deduce Theorem \ref{t-classification-actions}, assuming Proposition \ref{t-hmap}. 

\begin{proof}[Proof of Theorem \ref{t-classification-actions}]
	Let $h\colon \R\to Z$ be the locally injective map provided by Proposition \ref{t-hmap} and set $z:=h(0)$. Since the quotient map $q\colon Y\to Z$ identifies each $\Phi$-orbit to a leaf, and maps exactly two $\Phi$-orbits to each leaf,   the map $h$ admits two lifts $\tilde{h}_1, \tilde{h}_2\colon \R \to Y$ taking values in two different $\Phi$-orbits, with $\tilde{h}_i\circ \hat{\sigma}=\tilde{h}_{1-i}$. Moreover, since $\hat{\sigma}$ reverses the time  of the flow $\Phi$, exactly one of these two lifts is orientation preserving with respect to the  orientation of orbits given by the flow. Rename this lift by $\tilde{h}$ and set $y=\tilde{h}(0)$. If $y$ is not  a periodic orbit, the map $\tilde{h}$ is a bijection of $\R$ onto the orbit, and it follows that $\rho$ is pointed conjugate to $\rho_y$. 
	
	The case where $y$ is a periodic orbit requires some further inspection. Assume that the period of $y$ is $m\in \N$. Let $H$ be the subgroup of $G$ of elements that act trivially on the leaf $\ell$ of $z$, and let $H_0 \le H$ be the subgroup consisting of elements $g$ such that $\overline{\supp_Z(g)} \cap \ell =\varnothing$. Note that both are normal subgroups of $G$: the group $H$ coincides with the group of elements $g\in G$ such that $\rho_y(g)$ is a translation by an integer multiple of $m$, while $H_0$ is the  kernel of $\rho_y$, using that non-periodic orbits of $\Phi$ are dense. The action of $G$ on $\ell$ identifies $G/H$ with the natural copy of $\Tsing$ acting on $\ell=\R/m\Z \cong \mathbb{S}^1$.   Accordingly, the  action $\rho_y$ induces an isomorphism of $G/{H_0}$ to the natural copy of $\tilde{\Tsing}$ conjugated by the dilation $x\mapsto mx$. 
	Now the map $h\colon \R \to \ell$ is a covering map, equivariant with respect to  the action $\rho$ and the natural action on $\ell$. Upon conjugating $\rho$, we can assume that $h$ coincides with the natural covering map  $\R$ to $\R/m\Z \cong \ell$, so that $\rho$ also takes values in the $m$-dilated copy of $\tilde{\Tsing}$. Since $H$ acts trivially on $\ell$, we deduce that $\rho(H)$ is central in $\rho(G)$, and in particular it is abelian. However, the subgroup $H_0$ is perfect, as a consequence of Proposition \ref{p-fragmentation}. Therefore $\rho(H_0)=\{\id\}$, and thus $\rho$ descends to a  map $\bar{\rho} \colon G/H_0\cong \tilde{\Tsing}\to \tilde{\Tsing}$ that satisfies the assumption of Lemma \ref{l-tilde-T}. We conclude that $\bar{\rho}$ is the identity and that $\rho=\rho_y$. \qedhere 
\end{proof}

The rest of the section is devoted to the proof of Proposition \ref{t-hmap}.
The proof requires some preliminary work. We keep assumptions as in the statement, and we introduce further notation. Let $H$ be a subgroup of $G$, and take $\xi\in\R$. When $\xi\in\supp(\rho\vert_H)$, we denote by $I^\rho(H,\xi)$ the component of the support of $\rho(H)$ containing $\xi$. When $\xi\in\fix(\rho\vert_H)$, we simply define $I^\rho(H,\xi)=\{\xi\}$. In the case that $H=G_{C,I}$ for some chart $Z_{C,I}$, we write $I^\rho(C,I,\xi)$ instead of $I^\rho(G_{C,I},\xi)$. When there is no risk of confusion, we simply write $I(H,\xi)$ or $I(C,I,\xi)$. We then have the following lemma, which is the main application of Proposition \ref{p-semicon2}.
\begin{lem}\label{l-minimal-CIxi}
	With assumptions as in Theorem \ref{t-classification-actions}, given any dyadic chart $Z_{C,I}$ and $\xi\in \R$, we have that the $\rho$-action of $G_{C, I}$ on $I(C, I, \xi)$ admits a perfect minimal set.
\end{lem}
\begin{proof}
	When $\xi\in \fix(\rho\vert_{G_{C,I}})$ there is nothing to prove. So we can assume that $I(C,I,\xi)$ is a non-empty open interval. The group $G_{C, I}$ is naturally identified with the group $\mathcal{C}(C, \Fsing_I)$. Since $\Fsing_I\cong \Fsing$ is finitely generated and perfect (Lemma \ref{l-fingenperfect}), Proposition \ref{p-semicon2} gives the desired conclusion.
\end{proof}

Given points $z\in Z$, $\xi\in\R$, and  a decreasing sequence of charts $\big(Z_{C_n,I_n}\big)_{n\in\N}$ satisfying $\bigcap_{n\in\N}Z_{C_n,I_n}=\{z\}$ (by metrisability of $X$, for any point $z\in Z$ we can find such a sequence), we write $I(z,\xi):=\bigcap_{n\in\N} I(C_n,I_n,\xi)$, which is the intersection of a nested sequence of intervals, and therefore an interval which contains $\xi$ (possibly reduced to the singleton $\{\xi\}$). 
Notice that the interval $I(z,\xi)$ does not depend on the decreasing sequence of charts considered.

\begin{lem}\label{l-intervals}
	With assumptions as in Theorem \ref{t-classification-actions}, we have $I(z,\xi)=\{\xi\}$ for every $z\in Z$ and $\xi\in \R$.
\end{lem}

\begin{proof}
	Consider the family of intervals $\mathfrak{I}:=\{I(z,\xi):z\in Z,\xi\in\R\}$. We first show that inclusion defines a partial order relation on $\mathfrak I$. For this, given intervals $I,J\subset \R$ (possibly singletons), we say that $I$ and $J$ are \emph{linked} if their interiors are neither disjoint, nor nested.
	\begin{claimnum}\label{cl-frakI-1}
		The family $\mathfrak I$ has no linked elements.
	\end{claimnum}
	\begin{proof}[Proof of claim]
		We distinguish two cases. Suppose first that $z_1=z_2=z$. In order to construct the intervals $I(z,\xi_1)$ and $I(z,\xi_2)$, we can use the same decreasing sequence $Z_{C_n,I_n}$ converging to $z$. Then, there are two possibilities:
		\begin{itemize}
			\item either $I(C_n,I_n,\xi_1)=I(C_n,I_n,\xi_2)$ for every $n\in\N$, in which case $I(z,\xi_1)=I(z,\xi_2)$, or
			\item for some $n\in\N$ it holds that $I(C_n,I_n,\xi_1)\cap I(C_n,I_n,\xi_2)=\emptyset$, in which case we get that $I(p,\xi_1)$ and $I(p,\xi_2)$ are disjoint.
		\end{itemize}
		In both cases we deduce that the corresponding intervals are unlinked. Assume now that $z_1\neq z_2$, and consider nested sequences of charts $(Z_{C_n,I_n})_{n\in\N}$ and $(Z_{D_n,J_n})_{n\in\N}$ converging to $z_1$ and $z_2$, respectively. 
		Suppose by contradiction that $I(z_1,\xi_1)$ and $I(z_2,\xi_2)$ are linked. Then, for sufficiently large $k\in\N$, we have that $I(C_k,I_k,\xi_1)$ and $I(D_k,J_k,\xi_2)$ are linked, while $Z_{C_k,I_k}$ and $Z_{D_k,J_k}$ are disjoint. The second condition implies that the subgroups $G_{C_k,I_k}$ and $G_{D_k,J_k}$ commute. This gives a contradiction since connected components of the support of commuting subgroups are pairwise unlinked.
	\end{proof}
	 We set now
	\[\textstyle \mathfrak{I}_0:=\left \{\mathsf{Int}\left (\bigcup_\alpha J_\alpha\right ):(J_\alpha)\text{ is a maximal chain of } (\mathfrak{I},\subseteq)\right \}.\]
	By Claim \ref{cl-frakI-1}, elements of $\mathfrak{I}_0$ are pairwise disjoint open intervals.
	\begin{claimnum}
		Every element of $\mathfrak{I}_0$ is a bounded interval.
	\end{claimnum}
	\begin{proof}[Proof of claim] Let $J_n=I(z_n, \xi)$ be an increasing sequence of intervals in $\mathfrak{I}$. Up to extracting a subsequence we can suppose that $(z_n)$ converges to some point  $z\in Z$.  Then for every chart $Z_{C,I}$ containing $z$, we eventually have $I(z_n,\xi)\subseteq I(C,I,\xi)$, so it is enough to prove that the latter is bounded for some chart $Z_{C, I}$. Choose a chart small enough so that we can find $h\in G$ such that $h(Z_{C,I})\cap Z_{C,I}=\varnothing$. Assume by contradiction that $I(C,I,\xi)$ is unbounded, and without loss of generality, we assume that it accumulates at $+\infty$. By Lemma \ref{l-minimal-CIxi}, $G_{C,I}$ admits a perfect minimal set $\Lambda\subset I(C,I,\xi)$, which also accumulate at $+\infty$. Since $I(C,I,\xi)$  is unbounded, every element in the centralizer of $G_{C,I}$ must preserve it.
		For $h\in G$  such that $h(Z_{C,I})\cap Z_{C,I}=\varnothing$, the subgroups $G_{C,I}$ and $hG_{C,I}h^{-1}$ commute. We deduce that $hG_{C,I}h^{-1}$ preserves $I(C,I,\xi)$, and, since it is perfect, we can apply Corollary \ref{c-fixinglambda} and deduce that it fixes $\Lambda$ pointwise. This implies that $hG_{C,I}h^{-1}$, and therefore also the conjugate $G_{C,I}$, have fixed points accumulating at $+\infty$. This gives the desired contradiction.
	\end{proof}
	
	Finally, we observe that for every $z\in Z$, $\xi\in \R$ and $g\in G$ we have that $\rho(g)(I(z,\xi))=I(g(z),\rho(g)(\xi))$, from which we deduce that the family $\mathfrak{I}$ is $\rho$-invariant, and therefore so is $\mathfrak{I}_0$. Putting all of this together, we have that the set given by the union of all the intervals in $\mathfrak{I}_0$ is a proper and open $\rho$-invariant subset. This implies that $\mathfrak{I}_0$ is empty, as desired.
\end{proof}

\begin{lem}\label{cor-main}
	With assumptions as in Theorem \ref{t-classification-actions}, take disjoint charts $Z_{C, I}$, $Z_{D,J}$, and a point $\xi\in\supp(\rho\vert_{G_{C,I}})$. Then, $G_{D,J}$ acts as the identity on $I(C,I,\xi)$. 
\end{lem}

\begin{proof}  
	As charts are disjoint, we have that the subgroups $G_{C,I}$ and $G_{D,J}$ commute. Thus, by Corollary \ref{c-fixinglambda}, it is enough to prove that the action of $G_{C,I}$ on $I(C,I,\xi)$ is minimal.
	To see this, we know by Lemma \ref{l-minimal-CIxi} that $G_{C,I}$ admits a perfect minimal set $\Lambda\subset I(C,I,\xi)$. Suppose by contradiction that $\Lambda\neq I(C,I,\xi)$. Note that $I(C,I,\xi)=I(C,I,\eta)$ for every $\eta\in I(C,I,\xi)$, so it is not restrictive to assume that $\xi\notin\Lambda$. Assuming so, take the connected component $U$ of $I(C,I,\xi)\setminus\Lambda$ containing $\xi$. By the choice of $U$, every element in $G_{C,I}$ either preserves $U$ or maps it disjointly. Therefore, for every family of subgroups $\{H_\alpha\}_{\alpha\in A}$ generating $G_{C,I}$, there must exist some $\alpha\in A$ such that $U\subset I(H_\alpha,\xi)$, otherwise we would have that $U$ is preserved by $G_{C,I}$, contradicting minimality of its action on $\Lambda$. Consider now a family of charts $\{Z_{D_\alpha,J_\alpha}\}_{\alpha\in A}$ contained in $Z_{C,I}$, covering $Z_{C,I}$. By the fragmentation lemma for $G$ (Proposition \ref{p-fragmentation}), we have that $G_{C,I}$ is generated by the subgroups $\{ G_{D_\alpha,J_\alpha}\}_{\alpha\in A}$. Therefore, we can find $\alpha\in A$ such that $U\subset I(D_\alpha,J_\alpha,\xi)$. Starting the argument again with the chart $Z_{D_\alpha,J_\alpha}$, and proceeding by induction, we can construct a nested sequence of charts $(Z_{C_n,I_n})_{n\in\N}$ such that $\bigcap_{n\in\N} Z_{C_n,I_n}$ is a singleton $\{z\}$, and such that $ U\subset I(C_n,I_n,\xi)$ for every $n\in\N$. We conclude that $I(z,\xi)$ contains $U$, but this contradicts Lemma \ref{l-intervals}.
\end{proof}

\begin{lem}\label{lem-exanduni}
	With assumptions as in Theorem \ref{t-classification-actions}, we have that for every $\xi\in\R$, there exists a unique $z_\xi\in Z$ such that $\xi\in \supp(\rho\vert_{G_{C,I}})$ for every chart $Z_{C,I}$ containing $z_\xi$. 
\end{lem}

\begin{proof}
	Suppose by contradiction that for every $z\in Z$, there exists a chart $Z_{C_z,I_z}$ containing $z$, and such that $\xi\in\fix(\rho\vert_{G_{C_z,I_z}})$. Thus, by compactness of $Z$, we can find finitely many points $z_1,\ldots,z_n\in Z$  whose associated charts cover $Z$. On the other hand, by the fragmentation lemma (Proposition \ref{p-fragmentation}), we have that $G$ is generated by the subgroups $G_{C_{z_1},I_{z_1}},\ldots,G_{C_{z_n},I_{z_n}}$. This implies that $\xi\in\fix(\rho)$, which is a contradiction. Thus, there must be at least one $z\in Z$ such that $\xi\in\supp(\rho\vert_{G_{C,I}})$ for every chart $Z_{C,I}$ containing $z$. In order to check uniqueness, suppose by contradiction that there exist $z_1\neq z_2$ with this property. Take disjoint charts $Z_{C_1,I_1}$ and $Z_{C_2,I_2}$ containing $z_1$ and $z_2$, respectively. By our assumption on $z_1$, Proposition \ref{cor-main} gives that $G_{C_2,I_2}$ acts as the identity on $I(C_1,I_1,\xi)$. This contradicts our assumption on $z_2$.
\end{proof}

We are ready to prove Proposition \ref{t-hmap}. 
\begin{proof}[Proof of Proposition \ref{t-hmap}]
	Consider the map $h\colon \R\to Z$ that associates to each $\xi\in\R$ the element $z_\xi\in Z$ given by Lemma \ref{lem-exanduni}. It follows directly from the definition that $h$ is equivariant. 
		
	\begin{claim}
		The map $h$ is continuous.
	\end{claim}
	\begin{proof}[Proof of claim]
		Take $\xi\in\R$ and  a chart $Z_{C,I}$ containing $z_\xi$. Notice that, in order to show continuity of $h$, it is enough to show that $z_\eta\in Z_{C,I}$ for every $\eta\in I(C,I,\xi)$. Suppose by contradiction this is not the case. and take a point $\eta\in I(C,I,\xi)$ and a chart $Z_{D,J}$ containing $z_\eta$, and  disjoint from $Z_{C,I}$. By definition of $z_{\eta}$, we have $\eta\in \supp(\rho\vert_{G_{D,J}})$, but Lemma \ref{cor-main} implies that $G_{D,J}$ acts trivially on $I(C,I,\xi)$. This provides the desired contradiction.
	\end{proof}
	We deduce from the claim that the image of $\R$ is contained in a single leaf $\ell=\pi_Z(\{x\}\times \R) \subset Z$. It remains to prove that $h$ is locally injective. Arguing by way of contradiction, fix a point $\xi_0$ such that for any open interval $U\ni \xi_0$, the restriction $h\vert_U$ is not injective. By continuity of $h$, we can choose $U$ sufficiently small, so that $h(U)$ is a proper subset of the leaf $\ell$. Consider two points $\xi_1<\xi_2$ in $U$, with $h(\xi_1)=h(\xi_2)=:z$. Notice that there must exist $\xi\in(\xi_1,\xi_2)$ such that $h(\xi)\neq z$, because otherwise equivariance of $h$ would allow us to construct a proper, non-empty, open invariant subset of $\R$, contradicting minimality of $\rho$. Since $h([\xi_1,\xi_2])$ does not cover $\ell$, we can take a chart $Z_{C,I}$ with the following properties:
	\[z\notin Z_{C,I},\quad \pi_Z(\{x\}\times I)\subset Z_{C,I},\quad h(\xi)\in \pi_Z(\{x\}\times I).\]
	With such a choice, we have that the orbit $\rho(G_{C,I})(h(\xi))$ is dense in $\pi_Z(\{x\}\times I)$, and therefore $\rho(G_{C,I})(h(\xi))\not\subseteq h((\xi_1,\xi_2))$.
	On the other hand, since $z\notin Z_{C,I}$, we can repeat the argument in the proof of the claim, and get that $\{\xi_1,\xi_2\}\in\fix(\rho\vert_{G_{C,I}})$.
	By equivariance of $\rho$, this implies that $\rho(G_{C,I})(h(\xi))\subset h((\xi_1,\xi_2))$, contradicting the choice of the chart $Z_{C,I}$.
\end{proof}

\section{From the Deroin space to the space of irreducible representations}

Throughout this section we let $(X, \varphi, \sigma)$ be a reversible {subshift} satisfying Assumption \ref{standing}, and set again $G=\Tsingsf(\varphi, \sigma)$. In this section we further study the space $\rep(G; \R)$ and prove the corollaries stated in the introduction. 

\subsection{Lifting converging sequences}
Recall that for every $y\in Y$ we denote by $\rho_y\in \rep(G; \R)$ the associated action given by \eqref{e-action}. By Corollary \ref{c-mainDeroin}, the map $y\mapsto \rho_y$ identifies $(Y, \Phi)$ with a Deroin space for $G$. 
We keep this identification as implicit, and denote  by 
\[r_Y \colon \rep(G; \R) \to Y\]
the retraction map as in Proposition \ref{l-rprop},
namely $r_Y(\rho)=y$ if $y$ is the unique point such that $\rho$ is pointed semi-conjugate to $\rho_y$. 

In order to derive stronger conclusions on $\rep(G; \R)$ than the general ones that follow from Propositions \ref{l-rprop} and \ref{p-deroin-correspondence}, we first need to give a closer look at properties of the map $r_Y$ when we specialize the discussion to the groups $G=\Tsingsf(\varphi, \sigma)$.  In this case we have the following, which says that the map $r_Y$ behaves like an open map in the direction transversal to the flow. 

\begin{prop} \label{p-retraction-open}
Let $\rho\in \rep(G;\R)$, and assume that $y_0:=r_Y(\rho)$ can be written in coordinates as  $y_0=\pi_Y(x_0, t_0)$ for $(x_0, t_0)\in X\times \R$. Let $\mathcal{U}$ be an open neighbourhood of $\rho$. Then there exists an open neighbourhood $V\subset X$ of $x_0$ such that $r_Y(\mathcal{U})$ contains $\pi_Y(V\times\{t_0\})$. 

\end{prop}
\begin{proof}
		Let $S\subset G$ be a finite generating subset, and take a neighborhood $\mathcal U$ of $ \rho\in \rep(G;\R)$. By definition of the compact-open topology, we can find $\varepsilon>0$ and a compact interval $\tilde K$ containing $0$ such that if $ \rho'\in \rep(G;\R)$ is such that 
	\begin{equation}\label{eq-compact-open}
		\sup_{\gamma\in S}| \rho(\gamma)(t)-\rho'(\gamma)(t)|<\varepsilon\quad \forall t\in \tilde K,
	\end{equation}
	then $ \rho'\in \mathcal U$. One can simply ask that \eqref{eq-compact-open} is satisfied for a sufficiently dense, finite collection of points $t\in \tilde K$. More precisely, we are going to use the following property.
	
	\begin{claim}
		There exists a finite subset $\tilde D\subset \R$ containing $0$ such that the following holds. Let $ \rho'\in \rep(G;\R)$ be such that $\rho'(\gamma)(t)=\rho(\gamma)(t)$ for every $\gamma\in S$ and $t\in \tilde D$. Then, {$\rho'\in \mathcal U$}.
	\end{claim}
	\begin{proof}[Proof of claim]	
		Fix $\delta>0$ such that for any $\gamma\in S$ and $s,t\in \tilde K$ satisfying $|s-t|<\delta$, we have $|\rho(\gamma)(s)-\rho(\gamma)(t)|<\varepsilon$, and choose a finite subset $\tilde D\subset \R$  which is $\delta/2$-dense in $\tilde K$. Given $t\in \tilde K$, take $t_-,t_+\in \tilde D$ such that $t_-<t<t_+$, and $|t_+-t_-|<\delta$. As actions preserve the orientation, for any $\gamma\in S$ the images $\rho(\gamma)(t)$ and $ \rho'(\gamma)(t)$ are contained in the interval $( \rho(\gamma)(t_-), \rho(\gamma)(t_+))=( \rho'(\gamma)(t_-), \rho'(\gamma)(t_+))$, whose length does not exceed $\varepsilon$. This implies that $| \rho(\gamma)(t)-\rho'(\gamma)(t)|<\varepsilon$ for every $\gamma\in S$, as wanted.
	\end{proof}
	In what follows, for $x\in X$, we will shorten  notation by writing $\rho_x$ instead of $\rho_{\pi_Y(x, t_0)}$ for the action given by  \eqref{e-action}. Let $h:\R\to \R$ be the pointed semi-conjugacy between $\rho$ and $\rho_{x_0}$ (which is a continuous map), and write $K=h(\tilde K)$. 
	
{By construction, there exists a clopen neighborhood $V$ of $x_0$ such that the translation cocycle $\tau_\gamma$ is constant on $\pi_Y(V\times \{t_0+t\})$ for every $t\in K$ and $\gamma\in S$. Hence for any $x\in V$, the $\rho_x$-action of every  $\gamma\in S$ on $K$ is the same as its $\rho_{x_0}$-action on $K$.}

We set $\tilde \Omega:=\rho(G)(\tilde D)$ and $\Omega_x:=\rho_x(G)(h(\tilde D))$ for every $x\in V$, which we consider as ordered subsets of $(\R,<)$.  Note that both these sets contain $0$, which we think of as a marked point. We let $<_x$ denote the lexicographic ordering of $\Omega_x\times \tilde \Omega$: given $s_1,s_2\in \Omega_x$ and $t_1,t_2\in \tilde \Omega$, one has
	\[
		(s_1,t_1)<_x (s_2,t_2)\Leftrightarrow \left\{\begin{array}{l}
			s_1<s_2,\text{ or}\\[.5em]
			s_1=s_2\text{ and }t_1<t_2. 
		\end{array}\right.
	\]
Notice that the following equivalence holds for every $\gamma_1,\gamma_2\in G$ and every $t\in\R$: 
{\begin{equation}\label{Feq1} \rho(\gamma_1)(t)<\rho(\gamma_2)(t)\Leftrightarrow\Big( h(\rho(\gamma_1)(t)),\rho(\gamma_1)(t)\Big)<_{x_0}\Big(\rho(\gamma_2)(t)),\rho(\gamma_2)(t)\Big)
\end{equation}}
The set $\Omega_x\times \tilde \Omega$ is also naturally marked at $(0, 0)$.	The order $<_x$ is invariant under the action $\psi_x$ of $G$ defined by $\psi_x(g)\colon (s,t)\mapsto (\rho_x(g)(s),\tilde\rho(g)(t))$. We can then consider the dynamical realisation $\rho'_x\colon G\to \homeo_0(\R)$ of the action $\psi_x$, which comes with an equivariant order-preserving embedding $j_x\colon (\Omega_x\times \tilde \Omega,<_x) \to (\R,<)$ (see \cite[Lemma 2.2.14]{BMRT}), which we assume to send the marked point to $0$. Because of the choice of $V$, for any $x\in V$, $t\in \tilde D$, and $\gamma\in S$, we have 
\begin{equation}\label{Feq2} h(\rho(\gamma)(t))=\rho_{x_0}(\gamma)(h(t))=\rho_x(\gamma)(h(t)).\end{equation}
{Therefore putting together (\ref{Feq1}) and (\ref{Feq2}), we get that for any $x\in V$, $t_1,t_2\in \tilde D$, and $\gamma_1,\gamma_2\in S$, the following inequalities are equivalent}:
	\begin{align*}
		&\rho(\gamma_1)(t_{1})<\rho(\gamma_2)(t_{2})\\
		\Leftrightarrow\,& (\rho_x(\gamma_1)(h(t_1)) , \rho(\gamma_1)(t_{1}))<_x(\rho_x(\gamma_2)(h(t_2)) , \rho(\gamma_2)(t_{2}))\\
		\Leftrightarrow\,& \psi_x(\gamma_1)(h(t_1) ,t_1 )<_x\psi_x(\gamma_2)(h(t_2) ,t_2 )\\
		\Leftrightarrow\,& \rho_x'(\gamma_1)(j_x (h(t_1),t_1 ))< \rho_x'(\gamma_2)(j_x (h(t_2),t_2) ).
	\end{align*}
	Thus, for given $x\in V$, we can take an action $\rho''_x$ pointed conjugate to $\rho_x'$ such that $\rho(\gamma)(t)=\rho''_x(\gamma)(t)$ for every $\gamma\in S$ and $t\in \tilde D$. We deduce from the claim that  $\rho''_x\in\mathcal{U}$ for every $x\in V$.
	We claim that {$r_Y(\rho''_x)=\pi_Y(x,t_0)$}, i.e.\  that $\rho_x$ and $\rho''_x$ are pointed semi-conjugate. To see this, it is enough to see that $\rho_x$ and $\rho_x'$ are. This comes from the fact that $\rho_x$ is pointed conjugate to the dynamical realisation of the induced action on $\Omega_x$ (this follows from \cite[Lemma 2.2.14]{BMRT}), and the latter is pointed semi-conjugate to $\rho_x'$ (by considering the order-preserving projection $\Omega_x\times \tilde \Omega\to \Omega_x$). \qedhere
\end{proof}

The following result allows to completely describe  the closure of semi-conjugacy classes in $\rep(G; \R)$ (or in other words which actions can be reached by perturbations) in terms of the closure of $\Phi$-orbits in $(Y, \Phi)$.
\begin{thm} \label{t-accumulation}
Let $\mathcal{F}\subset \rep(G; \R)$ be a family of representations, and fix $\rho\in \rep(G; \R)$. Then the following are equivalent.
\begin{enumerate}
\item \label{i-up} $\rho$  is accumulated by {representations} which are semi-conjugate to elements of $\mathcal{F}$;
\item \label{i-down} $r_Y(\rho)$ belongs to the closure of the union of  $\Phi$-orbits of points in $r_Y(\mathcal{F})$.
\end{enumerate} 
\end{thm}
\begin{proof}
The implication \eqref{i-up} $\Rightarrow$ \eqref{i-down} is a general consequence of Proposition \ref{l-rprop}. The converse follows from Proposition \ref{p-retraction-open}, indeed if \eqref{i-down} holds, and if $r_Y(\rho)=\pi_Y(x_0, t_0)$, then for every neighbourhood $V$ of $x_0$ the local transversal $\pi_Y(V\times\{t_0\})$ intersects the orbit of some point in $\mathcal{F}$, and hence every neighbourhood of $\rho$ intersects the semi-conjugacy class of an element of $\mathcal{F}$. \qedhere
\end{proof}
In the rest of the section we describe some applications of our results.

\subsection{Rigidity and flexibility}

\begin{cor}\label{c-rigid}
The following conditions on a representation $\rho\in \rep(G; \R)$ are equivalent
\begin{enumerate}
\item \label{i-locally-rigid} $\rho$ is locally rigid;
\item \label{i-rigid} $\rho$ is rigid;
\item \label{i-open} the $\Phi$-orbit of $r_Y(\rho)$ is an open subset of $Y$.
\end{enumerate}
\end{cor}
\begin{proof}
 The equivalence between \eqref{i-rigid} and \eqref{i-open} is a general consequence of Proposition \ref{p-deroin-correspondence}, and it is clear that  \eqref{i-rigid} implies \eqref{i-locally-rigid} . If \eqref{i-locally-rigid} holds, then by Theorem \ref{t-accumulation}, $r_Y(\rho)$ is an interior point of its $\Phi$-orbit, hence its whole $\Phi$-orbit is open. \qedhere

\end{proof}

\begin{cor}\label{c-dense-characterisation}A representation $\rho\in \rep(G; \R)$ has a dense {semi-conjugacy} class if and only if $r_Y(\rho)$ has a dense orbit. In particular if $(X, \varphi)$ is minimal, then all semi-conjugacy classes are dense in $\rep(G; \R)$.
\end{cor}
\begin{proof}
Apply Theorem \ref{t-accumulation} to $\mathcal{F}=\{\rho\}$ to see that if $r_Y(\rho)$ has a dense orbit, then the semi-conjugacy class of $\rho$ accumulates on every $\rho'\in \rep(G; \R)$. \qedhere \end{proof}

Next we introduce the following strong form of flexibility of representations.

\begin{dfn}\label{s-universally-flexible}
A representation $\varphi \in \rep(G; \R)$ is \emph{universally flexible} if every neighbourhood $U$ of $\varphi$ intersects non-trivially every semi-conjugacy class in $\rep(G; \R)$.
\end{dfn}
We have the following. 
\begin{cor} \label{c-flexible}
A representation $\rho \in \rep(G; \R)$ is universally flexible if and only if 
$(Y, \Phi)$ (equivalently $(X, \varphi)$) has a unique non-empty closed minimal invariant subset and $r_Y(\rho)$ belongs to it.  \end{cor}
\begin{proof}
If $r_Y(\rho)$ belongs to a unique minimal subset of $(Y, \Phi)$, then the orbit-closure of every point in $Y$ contains it, so that by Theorem \ref{t-accumulation}, $\rho$ is accumulated by representations in any given semi-conjugacy class. That this condition is necessary follows from Proposition \ref{l-rprop}. 
\end{proof}

These results can be used to produce examples with various prescribed properties, by constructing subshifts $(X, \varphi, \sigma)$ with the desired corresponding properties (for instance among closed invariant subshifts of the subshift of reduced words $X_{\text{red}}\subset A^\Z$ as in Example \ref{ex.reversible}). 
As an example we include the following.
\begin{cor}[Groups with rigid and universally flexible representations] \label{c-prescribed-rigidity}
Fix $n\in 2\N\cup\{\infty\}$. Then there exists a finitely generated group $G$ such that $\rep(G; \R)$ contains uncountably many semi-conjugacy classes and the following hold. \begin{enumerate}
\item There are exactly $n$ semi-conjugacy classes of rigid representations.
\item Every representation which is not rigid is universally flexible.

\end{enumerate}

\end{cor}

\begin{proof}
By Corollaries \ref{c-rigid} and \ref{c-flexible}, it is enough to find for each $n\in \N \cup \{\infty\}$ a reversible subshift $(X, \varphi, \sigma)$ so that $(X, \varphi)$ has a unique minimal closed invariant subset which is infinite, and exactly $2n$ orbits of isolated points. This is routine, we sketch a construction below. Consider the subshift of reduced words $X_{\text{red}}\subset A^\Z$ from Example \ref{ex.reversible}.

Choose first  an infinite  minimal subshift $M\subset X_{\text{red}}$ which is invariant under $\sigma$.  Choose also a sequence $x\in M$. Using that $M$ is infinite and minimal, we can write $M$ as the intersection of a strictly decreasing sequence $M_0\supsetneq M_1\supsetneq M_2\supsetneq \cdots$ of irreducible subshifts {of finite type}, which  can be assumed to be $\sigma$-invariant. 

By properties of subshifts of finite type, we can find for each $n$ a sequence $x_n\in M_n$ which has some infinite prefix and suffix which coincide with some prefix and suffix of $x$, respectively, and such that $x\notin M_{n+1}$ (for this it is enough that between the desired prefix and suffix there is some finite word which is allowed in $M_n$ but not in $M_{n+1}$). Note that every accumulation point of the $\varphi$-orbit of $x_n$ and of $\sigma(x_n)$ belongs to $M$. Then for $n\ge 0$ let 
\[X_n=X\cup\{\varphi^i(x_j), \varphi^i(\sigma (x_j)): i\in \Z, j\le n\}\]
and $X_\infty=\bigcup X_n$. These satisfy the desired conclusion, respectively for finite $n$ and for $n=\infty$. \qedhere
\end{proof}
\begin{rem}
The restriction that $n$ be even is necessary in Corollary \ref{c-prescribed-rigidity} since here we do not allow semi-conjugacies to reverse the orientation, so that rigid semi-conjugacy classes always come in pairs.
\end{rem}

\subsection{Cantor--Bendixson rank} As another application, let us consider a notion of Cantor--Bendixson rank for the space of semi-conjugacy classes in $\rep(G; \R)$. To this end, given a finitely generated group $G$ define a decreasing transfinite sequence of subspaces $\rep^\alpha(G; \R)$ invariant under semi-conjugacy, as follows:
\begin{itemize}
\item if $\alpha=\beta+1$ is a successor, we define $\rep^\alpha(G; \R)$ by removing from $\rep^\beta(G; \R)$ all semi-conjugacy classes that are open in $\rep^\beta(G; \R)$;
\item if $\alpha$ is a limit ordinal we define $\rep^\alpha(G; \R)=\bigcap_{\beta<\alpha} \rep^\beta(G; \alpha)$. 
\end{itemize}

If $(\mathcal{D}, \Psi)$ is a Deroin space for $G$, then a similar process yields a decreasing sequence of closed subsets $\mathcal{D}_\alpha$, obtained by successively removing open orbits, with $\rep^\alpha(G)=r_\mathcal{D}^{-1}(\mathcal{D}_\alpha)$. These sequences must stabilize at some countable ordinal, see \cite[Theorem 6.9]{Kechris}. 
\begin{dfn} \label{d-CB-rank}
Let $G$ be a finitely generated group. The smallest ordinal $\alpha$ such that $\rep^\alpha(G; \R)=\rep^{\alpha+1}(G; \R)$ (equivalently $\mathcal{D}_\alpha=\mathcal{D}_{\alpha+1}$) is called the \emph{semi-conjugacy CB-rank}  of $\rep(G; \R)$. 
\end{dfn}

Note that if $\alpha$ is the semi-conjugacy CB-rank, then $\rep^\alpha(G; \R)=\varnothing$ if and only if $\rep(G; \R)$ contains only countably many semi-conjugacy classes.

Recall that the usual Cantor--Bendixson rank (CB-rank) of a compact space is defined in a similar way, but by successively removing its isolated points. 
Then for the groups $G=\Tsingsf(\varphi, \sigma)$ we have the following.

\begin{cor} \label{c-CB-rank}
The semi-conjugacy CB-rank of $\rep(G; \R)$ is equal to the usual CB-rank of $X$.
\end{cor}
\begin{proof}
It is clear that open orbits of the suspension flow $(Y, \Phi)$ are removed according to the usual Cantor--Bendixson derivative process in $X$. \qedhere
\end{proof}

\begin{rem}
{We could not locate a reference stating which ordinals are realizable as the CB-rank of subshifts (although there are various such results for special classes of subshifts). However Ville Salo kindly communicated us a proof that the CB-rank of countable susbhifts are exactly the finite ordinals and the countable ordinals of the form $\beta+2$. The following slightly simpler version of his construction allows to realize all countable ordinals of the form  $\beta+3$.  On an alphabet of the form $A=B\sqcup\{\ast\}$, let $x$ be the constant sequence of $\ast$. Choose a countable closed subset $C\subset B^\N$, and let $(C_\alpha)$ be its sequence of CB-derivatives. Denote by $\alpha_0$ the CB-rank of $C$ and notice that $C_{\alpha_0}=\emptyset$ since $C$ is countable. We proceed to show how to construct a subshift with rank $\alpha_0+2$. The set $C$ can be chosen so that $\alpha_0=\beta+1$ for any given countable ordinal $\beta$ (for instance choosing $C$ to be homeomorphic to the ordinal $\beta+1$ with the order topology); this will show that any countable ordinal of the form $\beta+3$ can be the CB-rank of a subshift.}

{For each $c=(b_n)\in C$, let $x_c$ be the sequence in $A^\Z$ obtained by replacing the letter at position $2^n$ of $x$ with $b_n$ for $n\ge 0$. Write $X_1=\{\sigma^n(x_c):c\in C,\text{ }n\in\Z\}$ and let $X\subset A^\Z$ be the subshift given by the closure of $\{x\}\cup X_1$. It is not difficult to see, by construction of $X_1$, that $X$ is contained in the union $\{x\}\sqcup X_1\sqcup X_{2}$ where $X_{2}$ consists of those sequences with at most one letter in $B$. Notice also that the set $\{x_c:c\in B\}\subset X_1$ is a clopen subset of $X$ since it can be defined by looking at the cylinder associated to positions $1$, $2$, and $4$. Thus, it follows that the points $x_c$ are removed from $X$ according to the Cantor--Bendixson derivative process in $C$ and therefore $X_{\alpha_0}\subset \{x\}\cup X_2$. Moreover, it is direct to see that $X_{\alpha_0}$ meets both $\{x\}$ and $X_2$. Since the points of $X_2$ are isolated points of $X_{\alpha_0}$ that accumulate on $\{x\}$, it follows that $X_{{\alpha_0}+1}=\{\ast\}$. Therefore the CB-rank of $X$ is ${\alpha_0}+2$ as desired. }

{Notice that $X$ can be turned into a subshift satisfying Assumption \ref{standing} through the doubling construction in Example \ref{e-doubling}. }

\end{rem}

\subsection{Connectedness properties}  Recall from the introduction that we say that a representation $\rho\in \rep(G;\R)$ is \emph{path-rigid} if its path component in $\rep(G;\R)$ coincides with its semi-conjugacy class.

\begin{cor} \label{c-path-rigid-text}
The following property hold.

\begin{enumerate}
\item \label{i-path-components} All representations $\rho\in \rep(G; \R)$ are path-rigid.  
\item \label{i-connected-components} The space $\rep(G; \R)$ is connected if and only if the only $\varphi$-invariant clopen subsets of $X$ are $\varnothing$ and $X$.
\item \label{i-nowhere-connected} If $X$ has no isolated point,  then $\rep(G; \R)$ is nowhere locally connected.
\end{enumerate}
\end{cor}
\begin{proof}
\eqref{i-path-components} follows from Proposition \ref{p-deroin-correspondence}. To show  \eqref{i-connected-components} note that every partition of $Y$ into two open subsets must consist of $\Phi$-invariant subsets, and thus corresponds to a partition of $X$ into $\varphi$-invariant clopen sets.  To show \eqref{i-nowhere-connected}, suppose that $X$ has no isolated point, fix $\rho\in \rep(G; \R)$, and let $\mathcal{U}$ be a neighbourhood of $\rho$ whose image is contained in a chart $Y_{C, I}$, with $C\subset X$ clopen.  Then Proposition \ref{p-retraction-open} implies that the composition of $r_Y|_{\mathcal{U}}$ with the projection $Y_{C, I}\to C$ is open. Since $C$ is totally disconnected and has no isolated point, it follows that no open subset of $\mathcal{U}$ is connected. \qedhere
\end{proof}

\appendix
\section{Properties of the groups $\Tsingsf(\varphi, \sigma)$}\label{appendix}

Here we discuss some generating properties of the group $\Tsingsf(\varphi, \sigma)$. These results and proof are largely analogous to the results on the groups $\Tsf(\varphi)$ from \cite{MB-T}. The main technical difference does not come from considering type-$\nice$ maps, but rather from the fact that we have to deal with charts in the dihedral suspension; we will detail the proofs when  differences are relevant.
Throughout the section, $(X,\varphi, \sigma)$ is as in Assumption \ref{standing}, and we set $G=\Tsingsf(\varphi, \sigma)$.

\subsection{More about charts}  Say that  a chart $Z_{C, I}$ is \emph{extendable} if it is contained in a chart $Z_{C, J}$, with $\overline{I}\subset J$.

\begin{lem}\label{l-admissible-strips}
	Let $x\in X$. If $x$ is periodic, of minimal period $n$, and $I$ is an open interval such that $|I|<n$, then there exists a clopen neighbourhood $C$ of $x$ such that $Z_{C, I}$ is an extendable chart. When $x$ is not periodic, the same conclusion holds for every bounded interval $I\subset \R$.  \end{lem}
\begin{proof}
	The proof is similar to that of \cite[Lemma 4.10]{confined}. We consider the case where $x$ is periodic of minimal period $n$, the other case can be treated similarly. Take an interval $I\subset\R$ satisfying $|I|<n$. Consider the collection \[\Gamma=\{\gamma\in D_\infty\setminus \{\id\} : \gamma(I)\cap I\neq \varnothing\},\]
	which consists of finitely many elements.
	It $\gamma\in \Gamma$ is a translation, then it must be by some integer $k$, with $|k|\le |I|<n$. As $n$ is chosen as the minimal period of $x$, we deduce that for any translation $\gamma\in \Gamma$, we have $\gamma(x)\neq x$. On the other hand, if $\gamma\in \Gamma$ is a reflection, then Assumption \ref{standing} guarantees that $\gamma(x)\neq x$ (this is one of the equivalent conditions in Lemma \ref{l-flipping-involution}). By continuity, we can find a clopen neighbourhood $C\subset X$ of $x$ such that $\gamma(C)\cap C=\varnothing$ for any $\gamma \in \Gamma$. With this choice, we have that $\gamma(C\times I) \cap (C\times I)=\emptyset$ for all $\gamma\in D_\infty\setminus \{\id\}$. Therefore $\pi_Z(C\times I)=Z_{C,I}$ is a chart.	
	 Notice that if $J\supset\overline{I}$ is another open interval such that $|J|<n$, the same argument shows that $Z_{C,J}$ is a chart, which implies that $Z_{C,I}$ is an extendable chart.
\end{proof}

If $Z_{C,I}$ is an extendable chart with $I=(a, b)$, we have that $\partial {Z_{C, I}}$ is the disjoint union of the subsets $\pi_Z(C\times \{a\})$ and $\pi_Z(C \times \{b\})$.  We refer to these as the \emph{sides} of ${Z_{C, I}}$. 

\begin{lem}[Chart decomposition]\label{l-chart-decomposition}
	The space $Z$ can be written as a finite union $Z=\bigcup_{i=1}^k \overline{Z_{C_i, I_i}}$, where $Z_{C_i, I_i}$ are extendable charts, such that for every distinct $i, j\in \{1,\ldots,k\}$, the intersection $\overline{Z_{C_i, I_i}}\cap \overline{Z_{C_j, I_j}}$ is contained in a single side of ${Z_{C_i, I_i}}$ and of ${Z_{C_j, I_j}}$.
\end{lem}

\begin{proof}
	Consider the dyadic intervals $I_j=(j/4,(j+1)/4)$ with $j\in \{0,1,2,3\}$. By Lemma \ref{l-admissible-strips}, for every $x\in X$, there exists a clopen subset $C\subset X$ such that $Z_{C, I_j}$ is an extendable chart for $j\in \{0,1,2,3\}$. Thus, there exists a clopen partition $X=C_1\sqcup\cdots\sqcup C_m$ such that $Z_{i,j}:=Z_{C_i,I_j}$ is an extendable chart for any $i\in \{1,\ldots,m\}$ and $j\in \{0,1,2,3\}$. 
	Finally, since the action of $D_\infty$ on $\R$ is isometric and the intervals $I_j$ have length $1/4$, which is smaller than $1/2$, which is the length of a fundamental interval for this action, we conclude that whenever $(i,j)\neq(i',j')$, the intersection $\overline{Z_{i,j}}\cap \overline{Z_{i',j'}}$ is contained in a single side of ${Z_{i,j}}$ and of ${Z_{i',j'}}$. \qedhere 
\end{proof}

\subsection{Fragmentation property}
The goal of this subsection is to prove the fragmentation lemma (Proposition \ref{p-fragmentation}) for $\Tsing(\varphi,\sigma)$. We start with the fragmentation lemma for the group $\Fsing$.

\begin{lem}[Fragmentation lemma for $\Fsing$]\label{lem-fragmentationFsing}
	Let $I,I_1,\ldots,I_n$ be open dyadic intervals so that $I=\bigcup_{i=1}^nI_i$. Then $\Fsing_I=\langle\bigcup_{i=1}^n\Fsing_{I_i}\rangle$. 
\end{lem}
\begin{proof}
	Let us detail the proof in the case $n=2$, the general case follows by induction. Assume, without loss of generality, that $I_1=(a,c)$ and $I_2=(b,d)$, with $a<b<c<d$, and take a dyadic rational $x\in(b,c)$. As in the proof of Lemma \ref{l-fingenperfect}, write
	\begin{equation}\label{eq-generators}
		\Fsing_I=\langle F'_I,S_x,S_a,S_d\rangle,
	\end{equation}
	where $S_x\subset \Fsing_{(b,c)}$, $S_a\subset \Fsing_{(a,c)}$, and $S_d\subset \Fsing_{(b,d)}$ are finite subsets generating the groups of germs $\mathcal{D}_x$, $\mathcal{D}_a^+$, and $\mathcal{D}_b^-$, respectively. Analogously, we can write $\Fsing_{I_1}=\langle F'_{I_1},S_x,S_a,S_c\rangle$ and $\Fsing_{I_2}=\langle F'_{I_2},S_x,S_b,S_d\rangle$. Thus, we have that 
	\[\langle\Fsing_{I_1},\Fsing_{I_2}\rangle=\langle F'_{I_1},F'_{I_2},S_x,S_a,S_b,S_c,S_d\rangle.\]
	Since $F'_I=\langle F'_{I_1},F'_{I_2}\rangle$, the previous line and (\ref{eq-generators}) imply that $\Fsing_I=\langle \Fsing_{I_1},\Fsing_{I_2}\rangle$, as desired. 
\end{proof}

We next need a result on homeomorphisms of $Z$, whose proof is a refinement of the argument given for Lemma \ref{l-lifting-H0}.

\begin{lem}\label{lem-cocyclesupport}
	Let $h\in \Hsf_0(\varphi,\sigma)$. Then, there exists an isotopy $(h_s)_{s\in[0,1]}$ between $h$ and the identity which satisfies $\overline{\supp_Z(h_s)}=\overline{\supp_Z(h)}$ for every $s\in [0,1)$.  
\end{lem}
\begin{proof}
	For the isotopy, we will consider the one  defined in the proof of Lemma \ref{l-lifting-H0}. Explicitly, we take a lift $f\in \Hsf_0(\varphi)$ of $h$, consider the isotopy $(f_s)_{s\in [0,1]}\subset \Hsf_0(\varphi)$ between $f$ and the identity that satisfies $\tau_{f_s}=(1-s)\tau_f$, and then project it to an isotopy $(h_s)_{s\in[0,1]}\subset \Hsf_0(\varphi,\sigma)$ between $h$ and the identity. In order to prove the equality between supports of the elements in the isotopy we need the following.
	
	\begin{claim}
		Given $g\in \Hsf_0(\varphi)$, write $T_g=\{y\in Y:\tau_f(y)\neq 0\}$. Then $\overline{T_g}=\overline{\supp_Y(g)}$. 
	\end{claim}
	\begin{proof}[Proof of claim]
		Notice that if the leaf through $y\in T_g$ is non-closed, then $y\in\supp_Y(g)$. On the other hand, Assumption \ref{standing} implies that the points of $T_g$ with non-closed leaves are dense. Since $T_g$ is open, this implies that $\overline{T_g}\subseteq\overline{\supp_Y(g)}$. The reverse inclusion is trivial.
	\end{proof}
	
	Take now $h\in \Hsf_0(\varphi,\sigma)$, $f\in \Hsf_0(\varphi)$, and $(f_s)_{s\in[ 0,1]}$ as above. By the choice of the isotopy, we have that $T_{f_s}=T_f$ for every $s\in[0,1)$. Thus, the claim implies that $\overline{\supp_Y(f_s)}=\overline{\supp_Y(f)}$ for every $s\in[0,1)$. This implies $\overline{\supp_Z(h_s)}=\overline{\supp_Z(h)}$ for every $s\in [0,1)$, as desired. 
\end{proof}

We can now prove the fragmentation lemma for $\Tsing(\varphi,\sigma)$.

\begin{proof}[Proof of Proposition \ref{p-fragmentation}]
	The proof is somehow analogue to the discussion in \cite[Appendix A]{MB-T}. We will use that $\Hsf_0(\varphi)$ is a topological group, whose topology can be defined by looking at displacement of elements along leaves (more precisely, by looking at the uniform norm on the translation cocycles). Moreover, using Lemma \ref{l-lifting-H0}, we can identify  $\Hsf_0(\varphi, \sigma)$ to a closed subgroup of $\Hsf_0(\varphi)$.
	
	We first construct a family of charts subordinated to $\mathcal{U}$ that is well suited for our purposes. For this, choose a decomposition $Z=\bigcup_{i=1}^k\overline{Z_{C_i, I_i}}$ as in Lemma \ref{l-chart-decomposition}.
	After subdividing these into smaller charts if needed, we can suppose that each $\{Z_{C_i, I_i}\}_{i\in \mathcal I}$ is contained in some element of $\mathcal{U}$. Slightly enlarge each chart $Z_{C_i, I_i}$ to a dyadic chart $Z_{C_i, J_i}$, where $J_i$ is an $\varepsilon$-neighbourhood of $I_i$ satisfying:
	\begin{enumerate}
		\item\label{i} $Z_{C_i, J_i}$ is still contained in an element of $\mathcal{U}$ for every $i\in \mathcal I$;
		\item\label{ii} whenever $i,j\in \mathcal I$, with $i\neq j$, are such that $\overline{Z_{C_i, J_i}}\cap \overline{Z_{C_j, J_j} }\neq \varnothing$, then $Z_{C_i, J_i}\cap Z_{C_j, J_j}$ is a chart of the form $Z_{D_{ij}, L_{ij}}$, where $L_{ij}$ is an interval of the form $L_i=(t_{ij}-\varepsilon, t_{ij}+\varepsilon)$, so that $\overline{Z_{C_i, I_i}}\cap \overline{Z_{C_j, I_j}}= \pi_Z(D_{ij} \times \{t_{ij}\})$.
	\end{enumerate} 
	Set
	\[
	\mathcal I=\left \{i\in \{1,\ldots,k\}:\overline{Z_{C_i, I_i}} \cap \overline{\supp_Z(g)}\neq \varnothing\right \},
	\]
	and notice that $\{Z_{C_i, J_i}\}_{i\in \mathcal{I}}$ is an open cover of $\overline{\supp_Z(G)}$ subordinate to $\mathcal{U}$, and the charts $\{Z_{D_{ij}, L_{ij}}\}_{i,j\in \mathcal I, i\neq j}$ (which are also subordinate to $\mathcal{U}$) are pairwise disjoint and cover all sides of the  charts $\{{Z_{C_i, I_i}\}}$. To get the statement, it is therefore enough to show that the subgroup
	\[G_{g}:=\left \{h\in G:\overline{\supp_Z(h)}\subseteq\overline{\supp_Z(g)}\right \}\]
	is contained in the subgroup
	\[\textstyle K:= \left \langle \bigcup_{i\in \mathcal I} G_{C_i,J_i}\right \rangle.\]
	For this, we will prove that the elements in $G_g$ with small displacement belong to $K$, and then that every element in $G_g$ can be written as a product of elements in $G_g$ with small displacement.
	
	For the first step, take $h\in G_g$ which displaces any point $z\in Z$ of at most $\varepsilon/2$ along its leaf.  We have $h( \pi_Z(D_{ij} \times \{t_{ij}\}))\subset Z_{D_{ij}, L_{ij}}$ for any distinct $i,j\in \mathcal I$. 
	Since Thompson's group action on dyadic rationals is transitive, we can find an element $k$ which is a product of elements from the commuting subgroups $\{G_{D_{ij}, L_{ij}}\}_{i,j\in\mathcal I,i\neq j}$ (in particular, $k\in K$),
	such that $kh$ fixes
	\[\bigcup_{i, j\in \mathcal I,i\neq j} \pi_Z(D_{ij}\times\{t_{ij}\})=\bigcup_{i\in \mathcal I} \partial \overline{Z_{C_i, I_i}}.\]
	Hence $kh$ preserves each chart $Z_{C_i, I_i}$, and therefore it can be decomposed as a product of elements from the commuting subgroups $\{G_{C_i, I_i}\}_{i\in \mathcal I}$. In particular, we have that $kh\in K$, and therefore $h\in K$. 
	
	For the second step, consider an arbitrary element $h\in G_g$. By Lemma \ref{lem-cocyclesupport}, there exists an isotopy $(h_s)_{s\in [0,1]}$ between $h$ and the identity satisfying $\overline{\supp_Z(h_s)}=\overline{\supp_Z(h)}$ for every $s\in[0,1)$. Choose an increasing sequence $0=s_0<s_1<\cdots<s_{N-1}<s_N=1$, so that for every $i\in \{0,\ldots,N-1\}$, the element $f_i=h_{s_i}h^{-1}_{s_{i+1}}$ displaces any point of at most $\varepsilon/2$ along its leaf. Note that we have $h=f_0\cdots f_{N-1}$. Since $\overline{\supp_Z(f_i)}\subseteq\overline{\supp_Z(h)}$, we can reason as in the first step to deduce that each $f_i$ can be written as a product of elements in $\{H_{C_i,I_i}\}_ {i\in \mathcal I}$, where $H_{C_i,I_i}\subset \Hsf_0(\varphi,\sigma)$ is the subgroup of homeomorphisms supported on $Z_{C_i,I_i}$ of the form $(x,t)\mapsto(x,f_x(t))$. On the other hand, the classical fact that Thompson's group $F_{I}$ is dense in $\homeo_0(I)$ implies that $G_{C_i,I_i}$ is dense in $H_{C_i,I_i}$. Thus, for any $i\in \{0,\ldots,N-1\}$, we can approximate each $f_i$ (and hence $h$) by an element of $K$. This allows to find an element $k\in K$ such that $hk^{-1}$ displaces points by at most $\varepsilon/2$ along leaves. Hence, using the first step again, we get $hk^{-1}\in K$, and so $h\in K$.\qedhere
\end{proof}

\subsection{Finite generation}
Here we prove Theorem \ref{t:finitely_generated}, which says that when $(X,\varphi,\sigma)$ is a reversible subshift, the group $G$ is finitely generated.

 Given an open interval $I$, we denote by $\Fsing_I^c$ the subgroup of $\Fsing_I$ whose elements have their support compactly contained in $I$. Recalling the notation from \S \ref{ssc:centralizers},  given a chart $Z_{C,I}\cong C\times I$, we will write $\Fsing_{C,I}$ for the subgroup $(\Fsing_I)_C$ of $G_{C,I}$ of elements which act trivially on the factor $C$ and as a fixed element of $\Fsing_I$ on $I$. Similarly, we write  $\Fsing_{C,I}^c$ for the subgroup $(\Fsing_I^c)_C$.

The following lemma is the analogue of \cite[Lemma 4.7]{MB-T}.

\begin{lem}[Intersection lemma] \label{l-intersection}
	Consider charts $Z_{C, I}$ and $Z_{D, J}$, where $(C, I)$ and $(D, J)$ are such that $C\cap D\neq \varnothing$ and $I\cap J\neq \varnothing$. Then, the group $\langle \Fsing_{C, I}^c , \Fsing_{D, J}^c \rangle$ contains the subgroups $\Fsing_{C\cap D, I}^c$ and $\Fsing_{C\setminus D, I}^c$. 
\end{lem}
\begin{proof}
	For every interval $L\subset I\cap J$, the charts $Z_{C, L}$ and $Z_{D, L}$ are both well defined.  Choose $L$ that avoids the lattice of half integers $\frac{1}{2} \Z$. Then every non-trivial element $\gamma\in D_\infty$ must map $L$ disjointly from itself. We deduce that  $\gamma(C\times L)\cap(D\times L)=\varnothing$ for any such $\gamma$. It follows that $Z_{C, L} \cap Z_{D, L}=Z_{C\cap D, L}$. Taking commutators as in the proof of \cite[Lemma 4.7]{MB-T}, we find that $\Fsing_{C\cap D, L}^c\le \langle \Fsing_{C, I}^c , \Fsing_{D, J}^c \rangle$. Now, conjugating $\Fsing_{C\cap D, L}^c$ by an element of $\Fsing_{C, I}^c$, we conclude as in \cite[Lemma 4.7]{MB-T}.
\end{proof}

\begin{proof}[Proof of Theorem \ref{t:finitely_generated}]
	Since $\varphi:X\to X$ is a subshift, $X$ admits a clopen partition $C_1\sqcup\cdots \sqcup C_k$ whose $\varphi$-translates form a prebasis of the topology. Fix dyadic intervals $I_{-1}, I_0, I_1\subset \R$ of length $|I_j|<1$ containing respectively $-1, 0, 1$, and whose union covers an open neighbourhood of $[-1, 1]$.  Using Lemma \ref{l-admissible-strips}, for every $i\in \{1,\ldots,k\}$, and $x\in C_i$, we can find a clopen subset $D\subset C_i$  containing $x$ such that the charts $Z_{D, I_\omega}$ are admissible for $\omega\in \{-1, 0, 1\}$. Thus, upon refining the partition $\{C_i\}$, we can assume  that all charts $Z_{C_i, I_\omega}$ are admissible.  Let $H=\left \langle \bigcup_{i=1}^k\bigcup_{\omega\in \{-1,0,1\}}\Fsing_{C_i, I_\omega} \right \rangle.$ By Lemma \ref{l-fingenperfect}, $H$ is finitely generated. We want to prove that $H =\Tsingsf(\varphi, \sigma)$. 
	
	\begin{claim}
		Fix integers $m\le 0 \le n$,  and a sequence $(i_{j})_{j=m}^n\subset \{1,\ldots,k\}$ such that $D:=\bigcap_{j=m}^n \varphi^j(C_{i_j})$ is non-empty. Let $J$ be an interval contained in one of the intervals $I_\omega$, for $\omega\in \{-1, 0, 1\}$.  Then $\Fsing_{D, J}^c\le H$. 
	\end{claim}
	\begin{proof}[Proof of claim]
		Fix a dyadic interval $L\subset I_0$ containing 0, and small enough so that $L-1\subset I_{-1}$ and $L+1\subset I_1$. We begin by observing that if $D\subset C_i$ for some $i\in \{1,\ldots,k\}$, and if $\Fsing_{D, L}^c\le H$, then actually we have $\Fsing_{D, J}^c\le H$ for every dyadic interval $J$ as in the statement of the claim. Indeed, by conjugating $\Fsing_{D, L}^c\le H$ by  elements of $\Fsing_{C_i, I_0}^c$, we obtain the conclusion for $J\subset I_0$. Choose $J\subset I_0 \cap I_1$, and repeat the reasoning to obtain the conclusion for $J\subset I_1$, and argue in the same way for $J\subset I_{-1}$.
		
		We now proceed by induction on $m\le 0\le n$. If $n=m=0$, then $D=C_{i_0}$  and   $\Fsing_{D, L}^c\le \Fsing^c_{C_{i_0}, I_0} \le H$.  Assume that $n>0$. Set $D'=\bigcap_{j=m}^{n-1} \varphi^j(C_{i_j})$. By induction we have $\Fsing^c_{D', L}\le H$ and thus, since $L-1\subset I_{-1}$, we also have  $\Fsing^c_{D', L-1}\le H$. Now note that $\Fsing^c_{D', L-1}=\Fsing^c_{\varphi^{-1}(D'), L}$ by the chart identification rules. Since $\varphi^{-1}(D')\subset C_{i_1}$, we can iterate this reasoning, and we find that $\Fsing^c_{\varphi^{-j}(D'), L}$ for all $j\in \{1,\ldots, n\}$. In particular this holds true for $j=n$, and since $\varphi^{-n}(D')\cap C_{i_n}= \varphi^{-n}(D)$, the intersection lemma (Lemma \ref{l-intersection}) implies that $\Fsing^c_{\varphi^{-n}(D), L}\le H$. Now applying again the observation at the beginning of the paragraph, we have that $\Fsing^c_{\varphi^{-n+1}(D), L}=\Fsing^c_{\varphi^{-n}(D), L+1}\le H$, and iterating this reasoning $n$ times we find $\Fsing^c_{D, L}\le H$. The inductive step on $m$ is similar. 
	\end{proof}
	Now, let $Z_{C, J}$ be an arbitrary chart, and let us show that $\Fsing_{C, J} \le H$. Since $\Fsing_{C, J}$ is generated by its subgroups $\Fsing^c_{C, J'}$, where $J'\subset J$ is arbitrarily small, it is enough to show this for the subgroups $\Fsing^c_{C, J}$ such that $|J| \le \varepsilon$ for some given $\varepsilon$.  Moreover, using that $\Fsing^c_{C, J}=\Fsing^c_{\varphi^n(C), J-n}$, we can assume that $J$ intersects $[0, 1]$.  If $\varepsilon$ is sufficiently small, this implies that $J$ is entirely contained in $I_0$ or $I_1$. Since the subsets $C_i$ form a generating partition for $\varphi$, we have a partition $C=D_1\sqcup \cdots \sqcup D_k$, where each $D_i$ is as in the claim. This gives the inclusion $\Fsing^c_{C, J}\le \langle \bigcup_{i=1}^k \Fsing^c_{D_i, J}\rangle \le H$, as desired.
	
	As observed in the proof of Lemma \ref{l-minimal-CIxi}, the subgroup $G_{C,J}$ is identified with the group $\mathcal C(C,\Fsing_J)$, and it is therefore generated by its subgroups $\Fsing_{D,J}$, with $D\subset C$. Therefore, we have $G_{C,J}\le H$. The fragmentation lemma (Proposition \ref{p-fragmentation}), gives finally $G\le H$, as desired.
\end{proof}

\bibliography{biblio.bib}

\bigskip

{\small

\noindent\textit{Joaqu\'in Brum\\
	IMERL, Facultad de Ingenier\'ia, Universidad de la República, Uruguay\\
	Julio Herrera y Reissig 565, Montevideo, Uruguay\\}
\href{mailto:joaquinbrum@fing.edu.uy}{joaquinbrum@fing.edu.uy}

\smallskip

\noindent\textit{Nicol\'as Matte Bon\\
Universit\'e Claude Bernard Lyon 1\\
 CNRS\\
 	Institut Camille Jordan (ICJ UMR CNRS 5208)\\
	43 blvd.\ du 11 novembre 1918,	69622 Villeurbanne,	France\\}
\href{mailto:mattebon@math.univ-lyon1.fr}{mattebon@math.univ-lyon1.fr}

\smallskip

\noindent\textit{Crist\'obal Rivas\\
	Dpto.\ de Matem\'aticas\\
	Universidad de Chile\\
	Las Palmeras 3425, Ñuñoa, Santiago, Chile\\}
\href{mailto:cristobalrivas@u.uchile.cl}{cristobalrivas@u.uchile.cl}

\smallskip

\noindent\textit{Michele Triestino\\
	Institut de Math\'ematiques de Bourgogne (IMB, UMR CNRS 5584) \& Institut Universitaire de France\\
	Universit\'e de Bourgogne\\
	9 av.~Alain Savary, 21000 Dijon, France\\}
\href{mailto:michele.triestino@u-bourgogne.fr}{michele.triestino@u-bourgogne.fr}

}

\end{document}